\documentclass[11pt]{article}

\usepackage[margin=3cm]{geometry}

\usepackage{amsmath}
\usepackage{amssymb}
\usepackage{amsthm}
\usepackage[dvipsnames]{xcolor}
\usepackage{tikz}
\usetikzlibrary{shapes.misc}
\usepackage{mathtools}
\usepackage{enumitem}

\usepackage[labelfont=bf, skip=6pt]{caption}
\usepackage[boxed]{algorithm}
\usepackage[noend]{algpseudocode}
\usepackage{tabularx}
\usepackage{comment}

\newtheorem{theorem}{Theorem}
\newtheorem{lemma}[theorem]{Lemma}
\newtheorem{proposition}[theorem]{Proposition}
\newtheorem{corollary}[theorem]{Corollary}

\theoremstyle{definition}
\newtheorem{definition}[theorem]{Definition}
\newtheorem{remark}[theorem]{Remark}

\usepackage[pdfpagelabels,bookmarks=false]{hyperref}

\newcommand{\newresult}[1]{\textcolor{blue}{\boldmath #1}}

\def\Bscr{\mathcal{B}}
\def\Cscr{\mathcal{C}}
\def\Fscr{\mathcal{F}}
\def\Lscr{\mathcal{L}}
\def\Nscr{\mathcal{N}}

\def\Sscr{\mathcal{S}}

\def\Escr{\mathcal{E}}
\def\Vscr{\mathcal{V}}
\def\Gscr{\mathcal{G}}
\def\crossings{\textnormal{cr}}
\def\faces{\textnormal{faces}}
\definecolor{darkgreen}{rgb}{0.0, 0.8, 0.0}

\begin{document}

\title{Packing cycles in planar and bounded-genus graphs\footnote{A preliminary extended abstract appeared in the Proceedings of the 34th Annual ACM-SIAM Symposium on Discrete Algorithms (SODA 2023), 2069--2086}}
\author{Niklas Schlomberg\thanks{Research Institute for Discrete Mathematics and Hausdorff Center for Mathematics, University of Bonn} 
\and Hanjo Thiele\thanks{Greenplan GmbH, Bonn. This work was started at the Research Institute for Discrete Mathematics, University of Bonn} 
\and Jens Vygen\thanks{Research Institute for Discrete Mathematics and Hausdorff Center for Mathematics, University of Bonn. Part of this work was done while visiting FIM, ETH Zurich}}
\date{\vspace*{-8mm}}
\maketitle

\begin{abstract}
We devise constant-factor approximation algorithms for finding as many disjoint cycles as possible 
from a certain family of cycles in a given planar or bounded-genus graph.
Here disjoint can mean vertex-disjoint or edge-disjoint, and the graph can be undirected or directed. 
The family of cycles under consideration must satisfy two properties:
it must be uncrossable and allow for an oracle access that finds a weight-minimal cycle in that family for given nonnegative edge weights or 
(in planar graphs) the union of all remaining cycles in that family after deleting a given subset of edges.

Our setting generalizes many problems that were studied separately in the past. For example,
three families that satisfy the above properties are (i) all cycles in a directed or undirected graph, 
(ii) all odd cycles in an undirected graph, and (iii) all cycles in an undirected graph that contain precisely one
demand edge, where the demand edges form a subset of the edge set. 
The latter family (iii) corresponds to the classical disjoint paths problem in fully planar and bounded-genus instances. 
While constant-factor approximation algorithms were known for edge-disjoint paths in such instances, 
we improve the constant in the planar case and obtain the first such algorithms for vertex-disjoint paths.
We also obtain approximate min-max theorems of the Erd\H{o}s--P\'osa type.
For example, the minimum feedback vertex set in a planar digraph is at most 12 times 
the maximum number of vertex-disjoint cycles.
\end{abstract}

\section{Introduction}

Let $G = (V,E)$ be a directed or undirected graph and let $\Cscr$ be a family of (simple) cycles in $G$.
We assume that $G$ is planar or embedded in an orientable surface of bounded genus.
We devise constant-factor approximation algorithms for finding a largest possible 
vertex-disjoint (or edge-disjoint) subset $\Cscr'\subseteq\Cscr$, i.e., no two cycles in $\Cscr'$ share a vertex 
(or edge, respectively).

The family $\Cscr$ must satisfy two properties.
First, it must be \emph{uncrossable}. Goemans and Williamson~\cite{GoeW98} defined:

\begin{definition}[uncrossable] \label{def:uncrossable}
A family $\Cscr$ of cycles in a graph is called \emph{uncrossable} if the following property holds.

Let $C_1,C_2\in\Cscr$ and $P_2$ a path in $C_2$ such that $P_2$ shares only its endpoints with $C_1$. 
Then there is a path $P_1$ in $C_1$ between these endpoints such that
$P_1+P_2\in\Cscr$ and $(C_1 - P_1)+(C_2 -P_2)$ contains a cycle in $\Cscr$.
\end{definition}

It is easy to see that the set of all cycles in a directed or undirected graph is uncrossable. 
In an undirected graph also the set of odd cycles and the set of $D$-cycles are uncrossable,
where $D\subseteq E$ and a $D$-cycle is a cycle that contains exactly one edge from $D$. 
For proofs and more examples we refer to Section~\ref{section:examples_of_uncrossable}.
In contrast, the set of even cycles is not uncrossable in general.

Second, we do not want to list all (possibly exponentially many) cycles in $\Cscr$ and hence need some
kind of implicit representation. Our algorithms will require either a \emph{weight oracle}
that computes an element of $\Cscr$ of minimum total weight for given nonnegative edge weights, 
or (in planar graphs) a \emph{support oracle}, returning the union of all remaining cycles in $\Cscr$ after deleting a subset of edges.
More precisely:

\begin{definition}[weight oracle, support oracle]
A \emph{weight oracle}, for a family $\Cscr$ of cycles in a graph $G$, 
takes as input nonnegative weights for the edges of $G$;
it outputs a cycle in $\Cscr$ that has minimum total weight.

A \emph{support oracle}, for a family $\Cscr$ of cycles in a graph $G=(V,E)$,
takes as input a subset $X \subseteq E$ of edges and outputs the union of the edge sets of all cycles
$C\in\Cscr$ that do not contain any edge of $X$.
\end{definition}
 
Here is our first main result:

\begin{theorem}\label{thm:main_combinatorial}
For any fixed $\epsilon>0$, there is a polynomial-time $(3+\epsilon)$-approximation algorithm
for each of the following problems.
Given a planar graph $G$ and a support oracle for an uncrossable family $\Cscr$ of cycles in $G$,  
\begin{itemize}
\item[(a)] find a maximum-cardinality vertex-disjoint subset of $\Cscr$;
\item[(b)] find a maximum-cardinality edge-disjoint subset of $\Cscr$.
\end{itemize}
\end{theorem}

This general result improves the approximation ratio in several special cases 
(e.g., for vertex-disjoint packing of odd cycles in undirected planar graphs
on the ratio 6 by Kr\'al\textquoteright{}, Sereni and Stacho~\cite{KraSS12})
and provides the first constant-factor approximation in others.
We will discuss applications and related work in detail in Section~\ref{sec:applications}, 
but mention one particularly interesting corollary already here.

In the \emph{maximum vertex-disjoint paths problem} we are given a graph $G=(V,E)$ and 
a subset $D\subseteq E$ (the set of demand edges) and ask for a maximum set of pairwise vertex-disjoint $D$-cycles. 
The \emph{maximum edge-disjoint paths problem} is defined analogously.
An instance $(G,D)$ is called \emph{fully planar} if $G$ is planar (to distinguish from the
case where only the supply graph $(V,E\setminus D)$ is planar).
Since the family of $D$-cycles in an undirected graph is uncrossable and has a polynomial-time support oracle 
(cf.\ Section~\ref{sec:oracles}), Theorem~\ref{thm:main_combinatorial} directly implies:

\begin{corollary}\label{cor:vdp_edp_planar}
For any fixed $\epsilon>0$,
there is a polynomial-time $(3+\epsilon)$-approximation algorithm for the 
\begin{itemize}
\item[(a)] maximum vertex-disjoint paths problem 
\item[(b)] maximum edge-disjoint paths problem 
\end{itemize}
in undirected fully planar instances.
\end{corollary}

Previously, no constant-factor approximation algorithm was known for vertex-disjoint paths. 
For the easier maximum edge-disjoint paths problem, constant-factor approximation algorithms were discovered only recently 
by Huang et al.~\cite{HuaMMSV21} and Garg, Kumar and Seb\H{o} \cite{GarKS22};
Corollary~\ref{cor:vdp_edp_planar}\,(b) improves on the approximation factor 4 from \cite{GarKS22}.
At the end of this paper we will also devise the first constant-factor approximation algorithm for the weighted vertex-disjoint paths
problem (see Theorem~\ref{thm:weighted_disjoint_paths_planar}).

The 4-approximation algorithm of \cite{GarKS22} works by rounding an LP solution 
and can be generalized to arbitrary uncrossable families. 
It begins by finding an optimum solution with laminar support (i.e., the interiors of two cycles do not cross; see Section~\ref{section:graphembedding})
to the natural linear programming relaxation 
\begin{equation}\label{eq:lp_edgedisjoint}
\max \left\{ \sum_{C\in\Cscr}x_C : \sum_{C\in\Cscr: e\in C}x_C\le 1 \ (e\in E),\ x_C\ge 0 \ (C\in\Cscr) \right\}.
\end{equation}
This is relatively easy for $D$-cycles (due to Seymour's \cite{Sey81} reduction to $T$-cut packing) 
but significantly harder for general uncrossable families.
We show that it can be done in general 
(Theorem~\ref{thm:compute_laminar_lp_solution}); then the
remaining algorithm of \cite{GarKS22} works without changes and yields part (b) of 
Theorem~\ref{thm:main_lpbased_planar} below. 
However, the methods of \cite{GarKS22} and \cite{HuaMMSV21} cannot be extended to vertex-disjoint packing, 
which seems to be fundamentally harder.

The LP relaxation of the natural integer programming formulation for vertex-disjoint cycle packing is:
\begin{equation}\label{eq:lp}
\max \left\{ \sum_{C\in\Cscr}x_C : \sum_{C\in\Cscr: v\in C}x_C\le 1 \ (v \in V),\ x_C\ge 0 \ (C\in\Cscr) \right\}
\end{equation}
Again we can obtain an LP solution with laminar support (Theorem~\ref{thm:compute_laminar_lp_solution}). 
Based on this and a new structural lemma (Lemma~\ref{lemma:efficient_cycle_lemma_planar})
we devise a greedy rounding algorithm that yields part (a) of the following theorem, our second main result:

\begin{theorem}\label{thm:main_lpbased_planar}
Given a planar graph $G$ and a weight oracle for an uncrossable family $\Cscr$ of cycles in $G$, we
can find a 
\begin{itemize}
\item[(a)] vertex-disjoint subset of $\Cscr$ whose cardinality is at least $\frac{1}{5}$ the value of the LP~\eqref{eq:lp}
\item[(b)] edge-disjoint subset of $\Cscr$ whose cardinality is at least $\frac{1}{4}$ the value of the LP~\eqref{eq:lp_edgedisjoint}
\end{itemize}
in polynomial time.
\end{theorem}

These approximation guarantees are weaker than the one in Theorem~\ref{thm:main_combinatorial}, but
Theorem~\ref{thm:main_lpbased_planar} yields approximate min-max theorems of the Erd\H{o}s--P\'osa type,
comparing the cycle packing number with the cycle transversal number (cf.\ Section~\ref{section:erdosposa}).
For instance, our Theorem~\ref{thm:main_lpbased_planar}\,(a) improves the upper bound on the integrality
gap of \eqref{eq:lp} for cycle packing in planar digraphs from 15.95 
(due to Cames van Batenburg, Esperet and M\"uller \cite{CamEM17}) to 5.
Together with an upper bound on the dual integrality gap 
(Goemans and Williamson~\cite{GoeW98}, Berman and Yaroslavstev~\cite{BerY12}), 
this implies that the feedback vertex number (the minimum number of vertices hitting all cycles) 
in any planar digraph is at most 12 times the maximum number of vertex-disjoint cycles
(the previously known bound was $38.28$). 
See Corollary~\ref{cor:primaldual}.

Huang et al.~\cite{HuaMMV21} devised a constant-factor approximation algorithm 
for edge-disjoint cycle packing in bounded-genus graphs; more precisely the maximum edge-disjoint paths problem. 
By combining their techniques with ours, we can generalize their result and transfer it also to vertex-disjoint packing:

\begin{theorem}\label{thm:main_lpbased_boundedgenus}
Given a graph $G$ embedded in a fixed orientable surface of genus $g$ 
and a weight oracle for an uncrossable family $\Cscr$ of cycles in $G$, we can find a 
\begin{itemize}
\item[(a)] vertex-disjoint subset of $\Cscr$ whose cardinality is 
$\Omega(\frac{1}{g^2\log g})$ times the value of the LP~\eqref{eq:lp}
\item[(b)] edge-disjoint subset of $\Cscr$ whose cardinality is 
$\Omega(\frac{1}{g^2\log g})$ times the value of the LP~\eqref{eq:lp_edgedisjoint}
\end{itemize}
in polynomial time.
\end{theorem}

For the disjoint paths problem, this extends to a weighted generalization; see Theorem~\ref{thm:weighted_disjoint_paths_bounded_genus}.

\subsection{Outline of our key technical ideas}

In this informal outline we focus on the (more difficult) vertex-disjoint cycle packing problem.
The proofs for edge-disjoint packing are similar or easier.

\medskip
The proof of Theorem~\ref{thm:main_combinatorial} consists of two main components.
We fix a planar embedding of $G$ and consider the \emph{face-minimal} cycles of $\Cscr$; 
after deleting redundant edges those are the cycles in $\Cscr$ that bound a finite face (because $\Cscr$ is uncrossable).
The face-minimal cycles do not ``cross'' other cycles in $\Cscr$.

The first component is a linear-time approximation scheme for finding as many vertex-disjoint face-minimal cycles as possible.
This is relatively easy using Baker's \cite{Bak94} technique and yields a set $\Lscr_1$.
After this, our algorithm continues by removing all vertices that appear in cycles of $\Lscr_1$ 
and simply recurses on the resulting smaller graph to obtain a set $\Lscr_2$; the output is $\Lscr_1\cup\Lscr_2$.

The key observation that leads to the approximation guarantee is that a cycle in $\Lscr_1$ shares vertices with
at most $3+\epsilon$ cycles of an optimal cycle packing $\Lscr^*$ on average. 
We prove the bound using a tree representation of the laminar family $\Lscr_1\cup\Lscr^*$.
A cycle in $\Lscr_1$ can share vertices only with its siblings and its parent in the tree representation,
and if a cycle in $\Lscr_1$ has many siblings that are cycles in $\Lscr^*$, 
the tree representation will have many leaves that are in $\Lscr^*$,
and we may assume that these are face-minimal cycles. 
Due to our approximation scheme, there are not much more of those than cycles in $\Lscr_1$.

\medskip
Proving the LP-based results (Theorem~\ref{thm:main_lpbased_planar} and \ref{thm:main_lpbased_boundedgenus})
is more difficult. 
The first step, computing an optimum solution to the LP~\eqref{eq:lp}, is easy with the ellipsoid method. 
Next we need to uncross the support of that LP solution. 

Generalizing the uncrossing technique of Huang et al.\ \cite{HuaMMV21}
we can obtain a near-optimal solution $x$ to~\eqref{eq:lp} such that any two cycles $C_1$ and $C_2$ with 
$x_{C_1}>0$ and $x_{C_2}>0$ cross at most once. 
To this end, we need a stronger uncrossing property: when $C_1$ and $C_2$ cross at $v$ and at $w$,
we can always reassemble the four $v$-$w$-paths to obtain two cycles $C_1'$ and $C_2'$ in $\Cscr$ (and possibly
other cycles that we do not consider any further) that cross neither at $v$ nor at $w$.
We prove that every uncrossable family of cycles has this property.

In planar graphs, cycles that cross at most once do not cross at all and thus form a laminar family.
For Theorem~\ref{thm:main_lpbased_planar}, we actually need an optimum LP solution, not just a near-optimal one,
which requires a very careful uncrossing procedure. After starting with an optimum LP solution whose
cycles are (on average) shortest possible, we can apply Karzanov's~\cite{Kar96} polynomial-time uncrossing algorithm.

Once we have an LP solution with uncrossed support (in the planar case: laminar support), we will apply a 
kind of greedy rounding algorithm to obtain an integral solution (i.e., a vertex-disjoint cycle packing).
To this end, we show that in every laminar set $\Lscr$ of cycles there exists one cycle $C^*$ and a small set $W$ 
of its vertices such that all the other cycles in $\Lscr$ either contain a vertex of $W$ or contain no vertex of $C^*$.
We will show $|W|\le 5$ if $G$ is planar; otherwise $|W|$ will depend linearly on the genus of the graph.
Once we have this ``Efficient Cycle Lemma'', the remaining proof is very easy, at least in the planar case: 
just include $C^*$ in the
final solution and remove all cycles that share a vertex with $C^*$; this decreases the LP value by at most $|W|$.
In the bounded-genus case we can do exactly the same for the separating cycles and follow \cite{HuaMMV21} for
the non-separating cycles.

The core of the proof is the Efficient Cycle Lemma. We will actually show that one can choose $C^*$
so that one of the two sides of $C^*$ in the embedding does not contain a side of any other cycle in $\Lscr$
(such cycles are called \emph{one-sided}, a notion that is very similar to face-minimal).
The main part of the proof of our Efficient Cycle Lemma consists of constructing an auxiliary graph $G'$ 
whose vertices are the one-sided cycles and embedding that graph in the same surface as $G$. 

To construct $G'$, we consider for each one-sided cycle $C$ 
a suitably chosen set $W(C)$ of vertices hitting all ``neighbour cycles''
and add an edge incident to $C$ for each $w\in W(C)$. The other endpoint of this edge will
be a one-sided cycle ``inside'' a carefully chosen neighbour cycle that meets $C$ at $w$.
Many edges of $G'$ can ``meet'' at a vertex $w$ of $G$, but nevertheless our construction will allow us 
to embed $G'$ in the same surface as $G$; moreover we 
can avoid homotopic edge pairs (parallel edges that bound an empty area homeomorphic to the disk). 
Therefore, by Euler's formula, $G'$ contains a low-degree vertex, which serves as $C^*$. 

The Efficient Cycle Lemma can also be combined with the fractional local ratio method
in order to deal with weighted cycle packing problems, in particular the maximum-weight disjoint paths problem.

\medskip
The rest of our paper is structured as follows. 
After some preliminaries in Section~\ref{sec:preliminaries},
we prove Theorem~\ref{thm:main_combinatorial} in Section~\ref{sec:combinatorial}.
In Section~\ref{sec:uncrossing} we discuss uncrossing procedures in detail; 
in particular we prove that Definition~\ref{def:uncrossable} implies 
a stronger uncrossing property (Theorem~\ref{thm:strongly_uncrossable}) 
and show how to uncross LP solutions.
This is needed for the proof of Theorems~\ref{thm:main_lpbased_planar} and~\ref{thm:main_lpbased_boundedgenus}, which can be found in Section~\ref{sec:lpbased}. 
In Section~\ref{sec:applications} we conclude with applications of our main theorems, 
in particular regarding the Erd\H{o}s--P\'osa property and the disjoint paths problem.
Section~\ref{sec:applications} also contains more pointers to related work.

\section{Preliminaries}\label{sec:preliminaries}

\subsection{Some uncrossable families of cycles}\label{section:examples_of_uncrossable}

In this paper, graphs can be undirected or directed; it will mostly be irrelevant or clear from the context.
All cycles are simple cycles.
We will sometimes consider cycles as (2-regular connected) subgraphs and sometimes as edge sets; this will not lead to ambiguities.
If we are given a graph $G=(V,E)$ and a subset $D\subseteq E$, 
then a \emph{cycle hitting $D$} is a cycle in $G$ that contains at least one edge of $D$,
and a \emph{$D$-cycle} is a cycle in $G$ that contains precisely one edge from $D$.

\begin{definition}[examples of uncrossable families]
Let $\Cscr_\textnormal{all}^{\rightarrow}$ be the set of all (directed) cycles 
and $\Cscr_\textnormal{girth}^{\rightarrow}$ the set of all shortest cycles in a given directed graph.
If an undirected graph is given, let
$\Cscr_\textnormal{all}$ denote the set of all cycles,
$\Cscr_\textnormal{girth}$ the set of all shortest cycles, and 
$\Cscr_\textnormal{odd}$ the set of odd cycles.
For an undirected graph $G=(V,E)$ and a subset $D\subseteq E$, let
$\Cscr_D^{\ge 1}$ denote the set of cycles hitting $D$, and
$\Cscr_D^{=1}$ the set of $D$-cycles.
\end{definition}

For several of these families, the following was already noted by \cite{GoeW98}:

\begin{proposition}\label{prop:uncrossable}
Let $G=(V,E)$ be a graph and $D\subseteq E$.
The families $\Cscr_{\textnormal{all}}^{\rightarrow}$, $\Cscr_\textnormal{girth}^{\rightarrow}$,
$\Cscr_\textnormal{all}$, $\Cscr_\textnormal{girth}$, $\Cscr_\textnormal{odd}$, $\Cscr_D^{=1}$, and $\Cscr_D^{\ge 1}$ 
are all uncrossable.
\end{proposition}

\begin{proof}
Let $\Cscr$ denote one of these seven families of cycles.
Let $C_1,C_2\in\Cscr$ and $P_2$ a path in $C_2$ such that $P_2$ shares only its endpoints with $C_1$.
Denote these endpoints by $v$ and $w$. These partition $C_1$ into two $v$-$w$-paths $P_1'$ and $P_1''$.

For $\Cscr_\textnormal{all}^{\rightarrow}$ and and $\Cscr_\textnormal{girth}^{\rightarrow}$, 
we take $P_1\in\{P_1',P_1''\}$ such that $P_1+P_2$ is a directed cycle.
For $\Cscr_\textnormal{all}$ and $\Cscr_\textnormal{girth}$, we can take $P_1\in\{P_1',P_1''\}$ such that $C_1-P_1\not=C_2-P_2$.
In both cases, $(C_1-P_1)+(C_2-P_2)$ is Eulerian and can be partitioned into cycles, not all of which are just pairs of parallel edges. If we started with shortest cycles, the resulting cycles have a total of at most $2g$ edges, where $g$ is the girth, and hence each of them has $g$ edges.

For $\Cscr_\textnormal{odd}$, we take $P_1\in\{P_1',P_1''\}$ such that $P_1+P_2$ 
contains an odd number of edges (note that $P_1'$ and $P_1''$ have different parity since $C_1$ is odd).
Then also $(C_1-P_1)+(C_2-P_2)$ contains an odd number of edges and can be partitioned into
cycles, at least one of which is odd.

For $\Cscr_D^{=1}$, we take $P_1\in\{P_1',P_1''\}$ such that $P_1+P_2$ (and hence also $(C_1-P_1)+(C_2-P_2)$)
contains exactly one edge of $D$ (note that exactly one of $P_1'$ and $P_1''$ contains an edge of $D$).
Again, $(C_1-P_1)+(C_2-P_2)$ can be partitioned into cycles, one of which is a $D$-cycle.

For $\Cscr_D^{\ge 1}$, choose $P_1\in\{P_1',P_1''\}$ such that the subsets of edges in $D$ 
differ for $P_1$ and $P_2$ and also for $C_1-P_1$ and $C_2-P_2$.
Again, $(C_1-P_1)+(C_2-P_2)$ can be partitioned into cycles, at least one of which contains an edge in $D$ and
is not just a pair of parallel edges.
\end{proof}

In contrast, the set of even cycles is not uncrossable (except in special cases such as bipartite graphs). 
In directed graphs, neither $D$-cycles nor odd cycles nor cycles hitting $D$ are uncrossable in general.

Another example of an uncrossable family that has been considered (e.g., by \cite{KraSS12}) is the set of \emph{$D$-odd} cycles: 
a cycle is $D$-odd if it contains an odd number of edges from $D$.
However, packing $D$-odd cycles reduces to packing odd cycles by subdividing all edges in $E\setminus D$; so we will not
consider $D$-odd cycles any further in this paper.

Goemans and Williamson \cite{GoeW98} also considered the set of cycles in an undirected graph $G=(V,E)$
that contain at least one vertex from a subset $S\subseteq V$. However, if we let $D$ be the set of all edges 
with at least one endpoint in $S$, then this family is $\Cscr_D^{\ge 1}$, so we do not need to consider this separately.

Packing shortest cycles (i.e., $\Cscr_\textnormal{girth}$) has been considered for example by \cite{RauR09}. 
Another uncrossable family of cycles with a relation to a network design problem was considered by \cite{BerY12}.
Probably there are further applications.

\subsection{Implementing the oracles}\label{sec:oracles}
 
The classical examples of uncrossable families offer weight and support oracles. First, let us see how
to implement the weight oracle.

\begin{proposition}\label{prop:weight_oracle}
Let $G=(V,E)$ be a graph and $D\subseteq E$. 
For each of the families $\Cscr_\textnormal{all}^{\rightarrow}$, $\Cscr_\textnormal{girth}^{\rightarrow}$, 
$\Cscr_\textnormal{all}$, $\Cscr_\textnormal{girth}$, $\Cscr_\textnormal{odd}$, $\Cscr_D^{=1}$, and $\Cscr_D^{\ge 1}$ 
there is a polynomial-time algorithm that implements the weight oracle. 
\end{proposition}

\begin{proof}
For the weight oracle, we are given nonnegative weights of the edges of $G$.
Finding a minimum-weight cycle or cycle hitting $D$ or $D$-cycle is easy, for example by applying 
Dijkstra's shortest path algorithm for finding a minimum-weight path from $s$ to $t$ in $(V,E\setminus\{e\})$ or $(V,E\setminus D)$ 
for each edge $e=\{t,s\}$ or $e=(t,s)$ in $E$ or $D$, respectively.
For $\Cscr_\textnormal{girth}^{\rightarrow}$ and $\Cscr_\textnormal{girth}$ we simply add a large constant to all weights
before calling the oracle.
Finding an odd cycle of minimum total weight reduces to weighted matching (cf.\ Section~29.11e of \cite{Sch03}).
\end{proof}

Now we implement the support oracle. 

\begin{proposition}\label{prop:compute_irrelevant_edges}
Let $G=(V,E)$ be a graph and $D\subseteq E$. 
For each of the families $\Cscr_\textnormal{all}^{\rightarrow}$, $\Cscr_\textnormal{girth}^{\rightarrow}$,
$\Cscr_\textnormal{all}$, $\Cscr_\textnormal{girth}$, $\Cscr_\textnormal{odd}$, $\Cscr_D^{=1}$, and $\Cscr_D^{\ge 1}$
there is a polynomial-time algorithm that implements the support oracle. 
\end{proposition}

\begin{proof}
Let $\Cscr$ be one of these seven families.
We show how to compute the set of edges that belong to at least one cycle in $\Cscr$.
Applying this to $G-X$ and the set $\Cscr[G-X]$ of cycles that do not contain any edge of 
the given edge set $X$ does the job.

For $\Cscr_\textnormal{all}^{\rightarrow}$ we compute the strongly connected components.
For $\Cscr_\textnormal{all}$ we can simply remove all bridges.
For $\Cscr_\textnormal{girth}^{\rightarrow}$ and $\Cscr_\textnormal{girth}$ we compute for every edge
(say $e=(t,s)$ or $e=\{t,s\}$) the distance from $s$ to $t$ in $G-e$. 
An edge $e=(t,s)$ or $e=\{t,s\}$ belongs to a cycle of length $g$ (where $g$ is the girth of $G$)
if and only if the distance from $s$ to $t$ in $G-e$ is $g-1$. 

For the other cases we consider the blocks (maximal 2-vertex-connected subgraphs) separately.
So assume that $G$ is 2-vertex-connected.

First consider $\Cscr_\textnormal{odd}$. 
By checking whether $G$ is bipartite we decide whether there is no odd cycle at all. Suppose there is an odd cycle $C$.
We claim that then every edge belongs to an odd cycle.
Let $e=\{v,w\}\in E$ be any edge that is not part of $C$.
Since $G$ is 2-vertex-connected, we can find two vertex-disjoint paths $P_v$ from $v$ to a vertex $p$ on $C$
and $P_w$ from $w$ to a vertex $q$ on $C$, with $p\not=q$. 
Now $C$ is partitioned into two $p$-$q$-paths, and exactly one of them is odd,
so exactly one them can be added to $P_v$, $e$, and $P_w$ to obtain an odd cycle that contains $e$.

Now consider $\Cscr_D^{\ge 1}$. There are again two cases: if $D=\emptyset$, then $\Cscr_D^{\ge 1}=\emptyset$.
Otherwise every edge $d\in D$ belongs to a cycle, and every pair of edges $e\in E$ and $d\in D$ belongs to a cycle
because $G$ is 2-vertex-connected. In other words, then every edge belongs to a cycle hitting $D$.

Finally consider $\Cscr_D^{=1}$. In this case we can consider the subgraphs induced by the vertex sets of the connected components of $(V,E\setminus D)$ separately because any $D$-cycle must lie completely inside such a component.
Therefore we can also consider the blocks of those subgraphs separately and hence we may assume that $G$ is $2$-vertex-connected and $(V, E \setminus D)$ is connected.
If now $D = \emptyset$, there are no $D$-cycles.
Otherwise, we claim that every edge in $G$ belongs to a $D$-cycle.

This is obvious for the edges in $D$. For an edge $e\in E\setminus D$ and some $d\in D$, there is a cycle  
that contains $e$ and $d$ because $G$ is 2-vertex-connected. 
Let $C$ be a cycle containing $e$, some edge from $D$, and as few edges from $D$ as possible.
Suppose $C$ is not a $D$-cycle.
Then $C\setminus D$ has at least two connected components.
Let $P$ be a minimal path with edges in $E\setminus D$ that connects the connected component of $C\setminus D$ containing $e$
with another connected component of $C\setminus D$. From $P$ and $C$ we can construct a cycle
that contains $e$ and fewer edges from $D$, but still at least one. This contradicts the choice of $C$. 
\end{proof}

Once we have a weight oracle or support oracle, we can do simple operations.
For instance, to apply the uncrossing property algorithmically, we also need to uncross two cycles. 
To this end, we note:

\begin{proposition}\label{prop:membership_with_supportoracle}
Given a graph $G$ and a weight oracle or support oracle for a family $\Cscr$ of cycles in $G$,
we can decide by one oracle call whether a given edge set $C$ contains the edge set of a cycle in $\Cscr$.
\end{proposition}

\begin{proof}
If we have a support oracle, we let $X$ consist of all edges that do not belong to $C$. 
Calling the support oracle for $X$ and checking whether the result is nonempty does the job.

If we have a weight oracle, set the weight of every edge outside $C$ to 1 and the weight of every edge of $C$ to 0.
Calling the weight oracle for these weights will produce a cycle with total weight 0 if and only if $C$ contains the edge set of a cycle in $\Cscr$. 
\end{proof}

This also yields a \emph{membership oracle} by applying Proposition~\ref{prop:membership_with_supportoracle}
to the edge set of a cycle. 

\subsection{Graph embedding}\label{section:graphembedding}

Throughout we also assume that $G$ is embedded in an orientable surface of fixed genus $g$; the sphere if $G$ is planar.
Recall that for fixed $g$ there is a linear-time algorithm that embeds a given graph in an orientable surface 
of genus $g$ or decides that there is no such embedding \cite{Moh99}.
We will often identify cycles with their embedding. 
Sometimes it will be convenient to fix one point $\infty$ of the surface that is outside the embedding of all vertices and edges.
The open connected components that result from removing all vertices and edges from the surface are called the
\emph{faces}. The face containing $\infty$ is called \emph{infinite}; all other faces are \emph{finite}.

If $G$ is planar (embedded in the sphere), 
all cycles are \emph{separating} (removing the embedding of the cycle disconnects the surface).
If $G$ is embedded in a higher-genus surface, we will also mainly deal with separating cycles 
(we will see that non-separating cycles can be handled as in \cite{HuaMMV21}).
Removing a separating cycle $C$ from the surface yields exactly two open connected components, 
which we call the two \emph{sides} of $C$.
The side that contains $\infty$ is called the \emph{exterior} of $C$, the other side the \emph{interior} of $C$.

Following \cite{GoeW98}, for a family $\Cscr$ of separating cycles and $C_1, C_2 \in \Cscr$ we say that $C_2$ \emph{contains} $C_1$ and write $C_1 \subseteq_{\infty} C_2$ if the interior of $C_2$ contains the interior of $C_1$.
The minimal cycles with respect to this partial order are called \emph{face-minimal}.

We say that two separating cycles $C_1$ and $C_2$ \emph{cross} 
if their interiors are not disjoint and none of the two interiors contains the other.
A multi-set of separating cycles is \emph{laminar} if no pair of its cycles crosses. 

Obviously, any vertex-disjoint set of separating cycles is laminar.
We remark that a laminar family of cycles is not necessarily uncrossable 
(e.g., if the family consists of only two cycles with disjoint interiors, two common vertices and no common edge).
However, the main point is that uncrossable families allow for an optimal (fractional) cycle packing in which the separating cycles form a laminar family.
We will see this in Section~\ref{sec:uncrossing}, but already note a folklore result for a special case here:

\begin{proposition}\label{prop:simple_laminar_edgedisjoint}
Let $\Cscr$ be an uncrossable family of cycles in a planar graph $G$. 
Then there is a maximum-cardinality edge-disjoint subset of $\Cscr$ that is laminar.
\end{proposition}

This is a special case of Proposition~\ref{prop:uncross_planar}, which we prove in Section~\ref{sec:uncrossing_planar}.

In planar graphs, the face-minimal cycles of an uncrossable family form a laminar family due to the following observations,
which follow directly from the definition: 

\begin{proposition}[\cite{GoeW98}]\label{prop:faceminimal_do_not_cross_and_are_faces}
Let $\Cscr$ be an uncrossable family of cycles in a planar graph. Then we have:
\begin{enumerate}[label=\textnormal{(\alph*)}]
\item \label{item:faceminimaldoesnotcross}
A face-minimal cycle $C\in\Cscr$ does not cross any cycle in $\Cscr$. 
\item \label{item:faceminimalarefaces}
If every edge of $G$ belongs to a cycle in $\Cscr$, all face-minimal cycles are given by boundaries of finite faces in $G$.
\hfill\qed
\end{enumerate}
\end{proposition}

\section{\boldmath{Packing cycles in planar graphs: $(3 + \epsilon)$-approximation}}\label{sec:combinatorial}

In this section we prove Theorem~\ref{thm:main_combinatorial}, i.e., for any fixed $\epsilon > 0$, we devise a polynomial-time $(3 + \epsilon)$-approximation algorithm 
for the problem of finding a maximum vertex-disjoint (or edge-disjoint) subset 
of an uncrossable family $\Cscr$ of cycles in a planar graph $G$, given by a support oracle.

We first concentrate on the vertex-disjoint case (Theorem~\ref{thm:main_combinatorial}\,(a))
and mention the very few changes that the edge-disjoint case requires in Section~\ref{sec:combinatorial_edge_disjoint}.

\subsection{Outline}\label{sec:outline_of_algorithm}

For a set $\Cscr'\subseteq\Cscr$ let $\Cscr'_{\min}$ denote the set of face-minimal cycles that belong to $\Cscr'$, 
and $V(\Cscr')$ the set of vertices that belong to a cycle in $\Cscr'$.
Our algorithm works as follows.
\begin{enumerate}
\item\label{step:preproc} 
Compute $\Cscr_{\min}$ (cf.\ Proposition~\ref{prop:compute_cmin}).
\item\label{step:facemin} 
Find a vertex-disjoint subset $\Lscr_1$ of $\Cscr_{\min}$; 
we will see (in Lemma~\ref{lem:PTAS_for_face_minimal_solution}) that we can maximize $|\Lscr_1|$ up to a factor $1-\epsilon$,
using a well-known technique by Baker~\cite{Bak94}.
\item\label{step:recurse} 
Remove all vertices in cycles of $\Lscr_1$ and recurse on the remaining graph to obtain a vertex-disjoint subset $\Lscr_2\subseteq\Cscr[G-V(\Lscr_1)]$.
\item Output $\Lscr_1\cup\Lscr_2$.
\end{enumerate}

Step~\ref{step:preproc} is easily implemented as follows.

\begin{proposition}\label{prop:compute_cmin}
 Given a planar graph $G$ and a support oracle for an uncrossable family $\Cscr$ of cycles in $G$, 
 we can compute all cycles in $\Cscr_{\min}$ in linear time.
\end{proposition}

\begin{proof}
 Call the support oracle with $X=\emptyset$ and remove all edges that the oracle does not return.
 Now every edge belongs to a cycle in $\Cscr$.
 Then, by Proposition~\ref{prop:faceminimal_do_not_cross_and_are_faces}\,\ref{item:faceminimalarefaces}, every face-minimal cycle bounds a finite face of the induced embedding. 
 So calling the membership oracle (Proposition~\ref{prop:membership_with_supportoracle}) 
 for each of the cycles bounding a finite face yields $\Cscr_{\min}$.
\end{proof}

\subsection{Approximation scheme for packing face-minimal cycles}

We now show how to implement Step~\ref{step:facemin}, which selects only face-minimal cycles. 
More precisely, we describe a linear-time approximation scheme 
for finding a maximum vertex-disjoint subset of $\Cscr_{\min}$, 
using that $\Cscr_{\min}$ is now given explicitly after Step 1 (Proposition~\ref{prop:compute_cmin}).
The basic idea is similar to the approximation scheme by Baker \cite{Bak94}, which is based on the notion of $k$-outerplanar graphs:

\begin{definition}[$k$-outerplanar] \label{def:outerplanarity}
 Let $G$ be a planar graph, embedded in the sphere. 
 We assign level 1 to all vertices that are on the boundary of the infinite face. 
 For $i \geq 1$, we assign level $i+1$ to all vertices that are on the boundary of the infinite face 
 after removing all vertices of level at most $i$. 
 $G$ is called \emph{$k$-outerplanar} if all vertices have level at most $k$.
\end{definition}

Using the well-known facts that any $k$-outerplanar graph has treewidth at most $3k-1$ \cite{Bod98}
and tree decompositions with small tree-width can be found efficiently \cite{RobS86,ArnCP87,Bod96},
Baker's technique~\cite{Bak94} can be described in a simpler form (the dependence of the runtime on $k$ is worse, but we will apply it only for
fixed $k=\lceil\frac{1}{\epsilon}\rceil$ anyway).

\begin{lemma}\label{lem:dynamic_program_for_outerplanar_graphs}
Let $k \in \mathbb{N}$ be a fixed constant.
Then there is a linear-time algorithm that, 
given a $k$-outerplanar graph $G$ and a family $\Cscr_{\min}$ of face-minimal cycles, 
computes a maximum vertex-disjoint subset of $\Cscr_{\min}$.
\end{lemma}

\begin{proof}
 Since $\Cscr_{\min}$ is given explicitly, we can delete all vertices and edges that do not belong to any cycle in $\Cscr_{\min}$; 
 then every cycle in $\Cscr_{\min}$ bounds a finite face 
 (cf.\ Proposition~\ref{prop:faceminimal_do_not_cross_and_are_faces}~\ref{item:faceminimalarefaces}).
 We extend $G$ to another planar graph $G^\prime$ by adding a new vertex inside each such face of $G$ 
 and connecting it to each vertex on the boundary of that face via a single edge.
 Let $V$ denote the set of vertices of $G$ and $V^\prime$ denote the set of new vertices (corresponding to the cycles in $\Cscr_{\min}$).
 It is easy to see that $G^\prime$ is still $2k$-outerplanar and therefore has treewidth at most $6k-1$ \cite{Bod98}.
 Thus, using Bodlaender's algorithm~\cite{Bod96}, we can compute 
 a rooted tree decomposition of width at most $6k-1$ for $G^\prime$ in linear time such that 
 each leaf bag contains exactly one vertex and each other bag $B$ fulfills one of the following properties:
 \begin{enumerate}
  \item $B$ has exactly two children in the tree decomposition, and they contain exactly the same set of vertices as $B$.
  \item $B$ has exactly one child in the tree decomposition, and the symmetric difference of $B$ and its child contains exactly one element.
 \end{enumerate}
The algorithm to compute a vertex-disjoint subset $\Sscr$ of $\Cscr_{\min}$ is given by a dynamic program 
that specifies for each vertex in $V^\prime$ whether the cycle corresponding to that face is in $\Sscr$ or not.
A candidate for a bag $B$ now specifies for each such vertex that is contained in some bag in the subtree rooted at $B$ whether the corresponding face is in $\Sscr$ or not 
(obeying the rule that $\Sscr \subseteq \Cscr_{\min}$ is a family of pairwise vertex-disjoint cycles).
 The label of that candidate indicates for each vertex in $B\cap V^\prime$ whether the corresponding face has been chosen to be in $\Sscr$ or not 
 and for each vertex in $B \cap V$ whether the vertex belongs to some cycle that has already been chosen to be in $\Sscr$.
 Now for each of the at most $2^{6k}$ possible labels at some bag $B$ we will compute a candidate maximizing the number of cycles chosen to be in $\Sscr$.
 For each of the above types of bags, it is easy to see that this can be done in constant time (for fixed $k$) if the optimum candidates for the children of $B$ are already computed.
 Therefore, going through the tree decomposition in reverse topological order and computing optimum candidates for all labels yields the result.
\end{proof}

Continuing like Baker~\cite{Bak94}, we obtain:

\begin{lemma}\label{lem:PTAS_for_face_minimal_solution}
 For every fixed $\epsilon > 0$ there is a linear-time algorithm that, 
 given a planar graph $G$ and a family $\Cscr_{\min}$ of face-minimal cycles in $G$, 
 computes a vertex-disjoint subset of $\Cscr_{\min}$ with at least 
 $(1-\epsilon) \textnormal{OPT}(\Cscr_{\min})$ elements, where $\textnormal{OPT}(\Cscr_{\min})$ 
 is the maximum cardinality of a vertex-disjoint subset of $\Cscr_{\min}$.
\end{lemma}

\begin{proof}
Let $k=\lceil\frac{1}{\epsilon}\rceil$.
We again remove all vertices and edges that do not belong to any cycle in $\Cscr_{\min}$;
then every cycle in $\Cscr_{\min}$ bounds a finite face.
For each integer $i > 0$ let $V_i$ be the set of vertices of level $i$ in $G$ as in Definition~\ref{def:outerplanarity},
and let $V_i:=\emptyset$ for $i\le 0$.
 For each integer $i$ define $G_i := G[\bigcup_{j=i}^{i+k-1} V_j]$.
 Since each of these subgraphs $G_i$ is $k$-outerplanar by definition, 
 we can find an optimum solution $\Sscr_i$ in $G_i$ 
 (i.e., a maximum vertex-disjoint subset of $\Cscr_{\min}[G_i]$) in linear time by  Lemma~\ref{lem:dynamic_program_for_outerplanar_graphs}.
 This defines feasible solutions $\Sscr^+_i := \bigcup_{j \in \mathbb{Z}} \Sscr_{i+jk}$ for $i=1,\ldots,k$.
 
 Because the vertices on the boundary of a face of $G$ differ by at most one in their level, 
each finite face of $G$ is contained in at least $k-1$ of the subgraphs $G_i$. 
 Therefore, the best solution among the $\Sscr^+_i$ ($i=1,\ldots,k$) 
 has at least $\frac{k-1}{k}$ times the size of an optimum solution.
 This proves the lemma.
\end{proof}

\subsection{Approximation guarantee \boldmath{$3+\epsilon$} (Proof of Theorem~\ref{thm:main_combinatorial}\,(a))}\label{section:3_approximation}

Let $0<\epsilon<\frac{1}{3}$ be a fixed constant.
By Proposition~\ref{prop:compute_cmin} and Lemma~\ref{lem:PTAS_for_face_minimal_solution},
Steps 1 and 2 of the algorithm described in Section~\ref{sec:outline_of_algorithm} can be implemented 
to run in linear time and yield a vertex-disjoint subset $\Lscr_1\subseteq\Cscr_{\min}$. 
After removing the vertices of these cycles (and all incident edges), we still have a support oracle
for the remaining cycles $\Cscr[G-V(\Lscr_1)]$ in the resulting graph $G-V(\Lscr_1)$.
Moreover, $\Cscr[G-V(\Lscr_1)]$ is uncrossable because $\Cscr$ is.
Hence we can recurse and conclude that the overall algorithm runs in quadratic time. 

It remains to prove the approximation guarantee.

\begin{lemma}\label{lemma:3_approximation}
 The algorithm described in Section~\ref{sec:outline_of_algorithm} computes a vertex-disjoint subset of $\Cscr$ 
 with at least $(\frac{1}{3} - \epsilon) \textnormal{OPT}(\Cscr)$ elements, 
 where $\textnormal{OPT}(\Cscr)$ denotes the maximum cardinality of a vertex-disjoint subset of $\Cscr$.
\end{lemma}

\begin{proof}
Let $G=(V,E)$ denote the given planar graph. We use induction on $|V|$.
Let $\Lscr^*$ be a maximum-cardinality vertex-disjoint subset of $\Cscr$. 
Clearly $\Lscr^*$ is laminar.
We may assume that every $\subseteq_{\infty}$-minimal cycle in $\Lscr^*$ is face-minimal 
(i.e., $\subseteq_{\infty}$-minimal in $\Cscr$); otherwise we could replace it by a $\subseteq_{\infty}$-smaller cycle. 

By combining Proposition~\ref{prop:compute_cmin} and Lemma~\ref{lem:PTAS_for_face_minimal_solution} we can find a vertex-disjoint subset $\Lscr_1 \subseteq \Cscr_{\min}$ with $|\Lscr_1| \ge (1-\epsilon)|\Lscr^*_{\min}|$ in polynomial time.
Moreover, $\Lscr^*\cup\Lscr_1$ is laminar by Proposition~\ref{prop:faceminimal_do_not_cross_and_are_faces}~\ref{item:faceminimaldoesnotcross}.
Consider a tree representation $T$ of $\Lscr^*\cup\Lscr_1$: an arborescence whose arcs represent the cycles in $\Lscr^*\cup\Lscr_1$ 
so that the arc representing $C_1$ is reachable from the arc representing $C_2$ if and only if $C_1\subseteq_{\infty} C_2$.
For $C\in\Lscr_1$ let $\textnormal{parent}(C)$ be the tail of the arc representing $C$. For a vertex $v\in T$ let 
$\delta^+_{\Lscr^*}(v)$ denote the set of arcs that represent a cycle in $\Lscr^*$ and leave $v$ in $T$, and
$\delta^-_{\Lscr^*}(v)$ denotes the (empty or one-element) set of arcs that represent a cycle in $\Lscr^*$ and enter $v$ in $T$.
See Figure~\ref{fig:treedecomp} for an example.
 
 \begin{figure}[htb]
  \begin{center}
   \begin{tikzpicture}[scale=0.33]
   \tikzstyle{l1}=[blue, very thick]
   \tikzstyle{l2}=[darkgreen, thick]
   \tikzstyle{B}=[darkgreen, thick]
   \tikzstyle{vertex}=[black,circle,fill,minimum size=5,inner sep=0pt,outer sep=1pt]
   \tikzstyle{edge}=[line width=1.5, ->]
   \begin{scope}[shift={(0, 0)}]
    \draw[B] (0, 0) circle (1);
    \draw[l2] (0, 0) circle (2);
    \draw[B] (0, 0) circle (3);
    \draw[B] (4, 5) circle (1);
    \draw[B] (6, 6) circle (1);
    \draw[B] (5, 5.5) circle (3);
    \draw[l1, dashed] (0, 0) circle (1);
    \draw[l1] (0.56, 4.48) circle (1.52);
    \draw[l1] (5.82, 3.85) circle (1.16);
    \draw[l2] (6,0) circle (2);
    \draw[B] (6,0) circle (1);
    \draw[l1, dashed] (6,0) circle (1);
    \draw[B] (2.5, 3) circle (8);
    \draw[l1] (11, 7) circle (1.39);
    \draw[B] (3.02, 3.25) circle (10.21);
    \draw[B] (11.73, 4.78) circle (0.95);
    \draw[B] (3.02, 3.25) circle (11.5);
    \draw[l1] (-6, -6) circle (1.42);
    \draw[B] (3.02, 3.25) circle (14.34);
    \draw[B] (-4.12, -7.5) circle (0.98);
    \node[B] at (-8.9, 11.9) {$a$};
    \node[l1] at (-7.5, -4.5) {$b$};
    \node[B] at (-3.1, -8.5) {$c$};
    \node[B] at (-6.6, 10.3) {$d$};
    \node[B] at (-5.6, 9.5) {$e$};
    \node[l1] at (9.6, 8.5) {$f$};
    \node[B] at (11.73, 3.3) {$g$};
    \node[B] at (-4, 8.4) {$h$};
    \node[B] at (2.5, 8) {$i$};
    \node[l1] at (-0.8, 5.8) {$j$};
    \node[B] at (2.41, -2.41) {$l$};
    \node[B] at (7.9, -1.7) {$k$};
    \node[B] at (7.2, -1.0) {$p$};
    \node[l1] at (7.2, -1.0) {$p$};
    \node[B] at (3, 6) {$m$};
    \node[B] at (5, 7) {$n$};
    \node[l1] at (4.5, 3) {$o$};
    \node[l2] at (1.71, -1.71) {$q$};
    \node[B] at (1, -1) {$r$};
    \node[l1] at (1, -1) {$r$};
   \end{scope}
   \begin{scope}[shift={(23, 17)}]
    \node[vertex] (root) at (0, 0) {};
    \node[vertex] (a) at (0, -4) {};
    \node[vertex] (b) at (-3, -8) {};
    \node[vertex] (c) at (0, -8) {};
    \node[vertex] (d) at (3, -8) {};
    \node[vertex] (e) at (3, -12) {};
    \node[vertex] (f) at (0, -16) {};
    \node[vertex] (g) at (3, -16) {};
    \node[vertex] (h) at (6, -16) {};
    \node[vertex] (i) at (1.5, -20) {};
    \node[vertex] (j) at (4.5, -20) {};
    \node[vertex] (k) at (7.5, -20) {};
    \node[vertex] (l) at (10.5, -20) {};
    \node[vertex] (m) at (-1.5, -24) {};
    \node[vertex] (n) at (1.5, -24) {};
    \node[vertex] (o) at (4.5, -24) {};
    \node[vertex] (p) at (7.5, -24) {};
    \node[vertex] (q) at (10.5, -24) {};
    \node[vertex] (r) at (10.5, -28) {};
    \draw[edge, B] (root) -- node[B, anchor=east]{$a$} node[cross out, red, draw, -]{} (a);
    \draw[edge, l1] (a) -- node[l1, anchor=east]{$b$} (b);
    \draw[edge, B] (a) -- node[B, anchor=east]{$c$} node[cross out, red, draw, -]{} (c);
    \draw[edge, B] (a) -- node[B, anchor=east]{$d$} node[cross out, red, draw, -]{} (d);
    \draw[edge, B] (d) -- node[B, anchor=east]{$e$} node[cross out, red, draw, -]{} (e);
    \draw[edge, l1] (e) -- node[l1, anchor=east]{$f$} (f);
    \draw[edge, B] (e) -- node[B, anchor=east]{$g$} node[cross out, red, draw, -]{} (g);
    \draw[edge, B] (e) -- node[B, anchor=east]{$h$} node[cross out, red, draw, -]{} (h);
    \draw[edge, B] (h) -- node[B, anchor=east]{$i$} node[cross out, red, draw, -]{} (i);
    \draw[edge, l1] (h) -- node[l1, anchor=east]{$j$} (j);
    \draw[edge, B] (h) -- node[B, anchor=east]{$k$} (k);
    \draw[edge, B] (k) -- node[B, anchor=east]{$p$} node[cross out, red, draw, -]{} (p);
    \draw[edge, l1, dashed] (k) -- node[l1, anchor=east]{$p$} node[cross out, red, draw, -]{}  (p);
    \draw[edge, B] (h) -- node[B, anchor=east]{$l$} node[cross out, red, draw, -]{} (l);
    \draw[edge, B] (i) -- node[B, anchor=east]{$m$} node[cross out, red, draw, -]{} (m);
    \draw[edge, B] (i) -- node[B, anchor=east]{$n$} node[cross out, red, draw, -]{} (n);
    \draw[edge, l1] (i) -- node[l1, anchor=east]{$o$} (o);
    \draw[edge, B] (l) -- node[B, anchor=east]{$q$} (q);
    \draw[edge, B] (q) -- node[B, anchor=east]{$r$} node[cross out, red, draw, -]{} (r);
    \draw[edge, l1, dashed] (q) -- node[l1, anchor=east]{$r$} node[cross out, red, draw, -]{}  (r);
   \end{scope}
   \end{tikzpicture}
  \end{center}
  \caption{The left-hand side shows an example of cycles in $\Lscr^*$ (\textcolor{darkgreen}{green}) and $\Lscr_1$ (\textcolor{blue}{blue and bold}). 
  The cycles $p$ and $r$ are contained both in $\Lscr_1$ and $\Lscr^*$. 
  The right-hand side shows the tree representation of $\Lscr^* \cup \Lscr_1$. 
  The cycles in $\Bscr$ are crossed out in red; these cycles are ``eliminated'' by $\Lscr_1$. 
  For the cycles corresponding to the first four levels of the tree, the bound for $|\Bscr|$ in the proof is tight.\label{fig:treedecomp}}
 \end{figure}
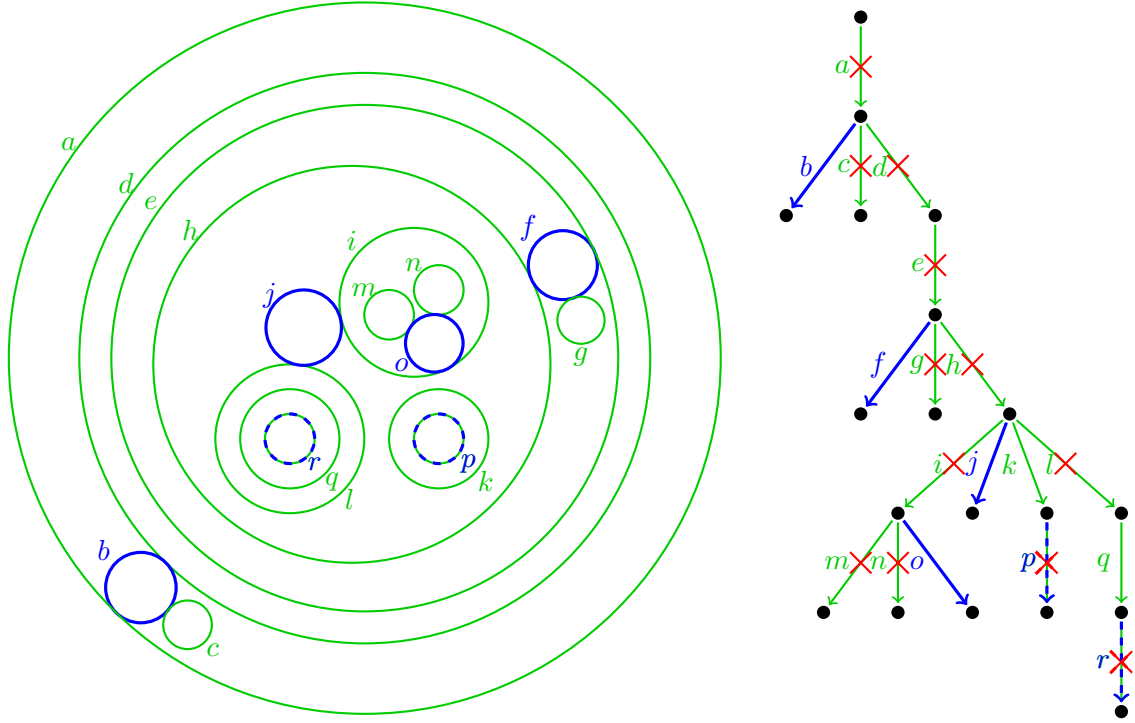
 
 Let  $\Bscr \subseteq\Lscr^*$ be the set of cycles in $\Lscr^*$ that have a vertex in common with at least one cycle in $\Lscr_1$ (these are the cycles from $\Lscr^*$ that do not exist anymore when we recurse on $G-V(\Lscr_1)$).
Note that the arc representing some $B\in \Bscr$ must be incident to $\textnormal{parent}(C)$ for some $C\in\Lscr_1$.
Hence
\begin{align*}
|\Bscr| &\ \le \ \left| \bigcup_{C\in \Lscr_1} \delta^-_{\Lscr^*}(\textnormal{parent}(C))\cup\delta^+_{\Lscr^*}(\textnormal{parent}(C)) \right| \\
&\ \le \  2|\Lscr_1| + \sum_{v\in T} \max \left\{0, |\delta^+_{\Lscr^*}(v)|-1 \right\} \\
&\ = \ 2|\Lscr_1| + |\Lscr^*_{\min}| - 1 \\
&\ \le \ 2|\Lscr_1| + \frac{1}{1-\epsilon} |\Lscr_1| - 1 \\
&\ \le \ \frac{1}{\frac{1}{3}-\epsilon} |\Lscr_1|,
\end{align*}
where we used in the equation that the $\subseteq_{\infty}$-minimal cycles in $\Lscr^*$ 
are face-minimal cycles (i.e., in $\Lscr^*_{\min}$). 

By the induction hypothesis, recursing on $G-V(\Lscr_1)$ yields a set $\Lscr_2$ of vertex-disjoint cycles with
$|\Lscr_2|\ge (\frac{1}{3}-\epsilon) |\Lscr^*\setminus\Bscr|$. Hence
\begin{align*}
|\Lscr_1\cup\Lscr_2| &\ = \ |\Lscr_1| + |\Lscr_2| \\
&\ \ge \  \left(\frac{1}{3}-\epsilon\right)|\Bscr| + \left(\frac{1}{3}-\epsilon\right) |\Lscr^*\setminus\Bscr| \\
&\ = \  \left(\frac{1}{3}-\epsilon\right) |\Lscr|^*. \qedhere
\end{align*}
 \end{proof}

This completes the proof of Theorem~\ref{thm:main_combinatorial}\,(a).
 
\subsection{The edge-disjoint case (Proof of Theorem~\ref{thm:main_combinatorial}\,(b))}\label{sec:combinatorial_edge_disjoint}
 
Since both the algorithm and the analysis are very similar to the vertex-disjoint case, we only discuss the differences to the previous sections. 
The algorithm still works as proposed in Section~\ref{sec:outline_of_algorithm}, with two modifications:

First, in Step~\ref{step:facemin} we want to find an edge-disjoint subset $\Lscr_1$ of $\Cscr_{\min}$. 
This can be done similarly to Lemma~\ref{lem:dynamic_program_for_outerplanar_graphs} and Lemma~\ref{lem:PTAS_for_face_minimal_solution}; 
however, they can be simplified since, after deleting redundant edges, edge-disjoint packings of face-minimal cycles in $G$ correspond to stable sets in (a subgraph of) the planar dual of $G$. 
Therefore we can directly use Baker's approximation scheme for the stable set problem in planar graphs \cite{Bak94}.

Second, in Step~\ref{step:recurse} we only delete the edges of $\Lscr_1$ instead of the vertices and recurse on $\Cscr[G-E(\Lscr_1)]$.

In order to adapt the analysis of the algorithm for the vertex-disjoint case to the edge-disjoint case, 
we need to ensure that $\Lscr^* \cup \Lscr_1$ is a laminar family of cycles. 
This is not trivial anymore since edge-disjoint cycles can cross. 
However, $\Lscr^*$ can be chosen to be laminar by Proposition~\ref{prop:simple_laminar_edgedisjoint};
then by Proposition~\ref{prop:faceminimal_do_not_cross_and_are_faces}\,\ref{item:faceminimaldoesnotcross} 
also $\Lscr^* \cup \Lscr_1$ is laminar.

If we then choose $\Bscr$ to be the set of cycles in $\Lscr^*$ that share an edge with at least one cycle in $\Lscr_1$, 
the analysis from Section~\ref{section:3_approximation} proves Theorem~\ref{thm:main_combinatorial}\,(b).

\section{Uncrossing}\label{sec:uncrossing}

As the name suggests, an uncrossable family of cycles allows for uncrossing.
We first review the simpler situation in planar graphs. 

\subsection{Uncrossing in planar graphs}\label{sec:uncrossing_planar}

In planar graphs, every uncrossable family has an optimal fractional cycle packing that is laminar. 
This is essentially shown in Lemma~4.2 of \cite{GoeW98}; for completeness we briefly re-prove it here: 

\begin{proposition}\label{prop:uncross_planar}
Let $G=(V,E)$ be a planar graph and $\Cscr$ an uncrossable family of cycles in $G$ given by a weight oracle. 
Given an explicit multi-subset $\Fscr\subseteq\Cscr$,
one can compute in polynomial time a laminar multi-subset $\Lscr\subseteq\Cscr$ with $|\Lscr|=|\Fscr|$ and such that
for every vertex $v$ and every edge $e$, the number of cycles that contain $v$ (or $e$) is no more in $\Lscr$ than in $\Fscr$.
\end{proposition}

\begin{proof}
For a cycle $C$ let $\faces(C)$ denote the set of faces of $G$ that are in the interior of $C$.
We initialize $\Lscr=\emptyset$ and iterate the following until $\Fscr=\emptyset$:

Pick a cycle $C_1$ from $\Fscr$.
While there is a cycle $C_2\in\Fscr$ that crosses $C_1$, 
we can find a path $P_2$ on $C_2$ inside the interior of $C_1$ that shares only its endpoints with $C_1$. 
Since $C_1,C_2\in\Fscr\subseteq\Cscr$ and $\Cscr$ is uncrossable, 
we can replace $C_1$ by a cycle $C_1^\prime$ (consisting of $P_2$ and a part of $C_1$)
with $\faces(C_1^\prime)\subset \faces(C_1)$ 
and $C_2$ by another cycle $C_2^\prime$ such that 
no vertex or edge is contained more often in $C_1^\prime$ and $C_2^\prime$ than in $C_1$ and $C_2$.
By Proposition~\ref{prop:membership_with_supportoracle} we can find such cycles using our weight oracle.

Iterate with $C_1^{\prime}$ in the role of $C_1$, note that its interior contains fewer faces.
Thus, after linearly many steps, the resulting cycle $C_1$ does not cross any cycle in $\Fscr$. We then remove
this cycle from $\Fscr$ and add it to $\Lscr$.
Throughout, we maintain the invariant that a cycle in $\Lscr$ does not cross any cycle in $\Lscr\cup\Fscr$.
\end{proof}

\subsection{Uncrossable families have a stronger uncrossing property}

We also want to uncross cycles in a bounded-genus graph, which is much more difficult.
In the special case of $D$-cycles, Huang et al.~\cite{HuaMMV21} showed how to uncross as much as possible. 
Their proof is implicitly based on a stronger uncrossing property:

\begin{definition}[strongly uncrossable] \label{def:strongly_uncrossable}
A family $\Cscr$ of cycles in a graph is called \emph{strongly uncrossable} if the following property holds.

Let $C_1, C_2 \in \Cscr$, and let $v$ and $w$ be two vertices that belong to both cycles $C_1$ and $C_2$. 
Then there are $v$-$w$-paths $P_1$ in $C_1$ and $P_2$ in $C_2$ such that both $P_1 + P_2$ and $(C_1 - P_1) + (C_2 - P_2)$ contain a cycle in $\Cscr$.
\end{definition}

It is obvious that strongly uncrossable families of cycles are uncrossable.
We will now show the nontrivial fact that the two definitions are in fact equivalent.
To this end, we call an edge set \emph{good} if it contains a cycle in $\Cscr$.
The following lemma is the core of the proof:

\begin{lemma}\label{lemma:strongly_uncrossable}
Let $\Cscr$ be an uncrossable family of cycles in a graph and $C_1,C_2\in\Cscr$. 
Let $v$ and $w$ be two vertices, $P_1$ a $v$-$w$-path in $C_1$, and $P_2$ a $v$-$w$-path in $C_2$.
Then $P_1+P_2$ or $P_1+(C_2-P_2)$ is good.
\end{lemma}

\begin{proof}
By induction on the number of edges in $P_1$.  

\emph{Case 1.}
If $v$ and $w$ are the only common vertices of $P_1$ and $C_2$, 
then (since $\Cscr$ is uncrossable) $P_1+P_2\in\Cscr$ or $P_1+(C_2-P_2)\in\Cscr$.

\emph{Case 2.}
Otherwise (by possibly swapping $P_2$ and $C_2-P_2$) we may assume 
that $P_1$ and $P_2$ have a common inner vertex $x$. 
Let $x$ be the first such vertex when traveling along $P_2$ from $v$.
Let $P_1'$ and $P_1''$ be the $v$-$x$-subpath and the $x$-$w$-subpath of $P_1$, and 
let $P_2'$ and $P_2''$ be the $v$-$x$-subpath and the $x$-$w$-subpath of $P_2$. 
(See Figure~\ref{fig:stronguncrossing}.)
Then $P_2'$ and $P_1''$ have no inner vertex in common. 

Suppose $P_1+P_2$ is not good.
First we apply the induction hypothesis to $C_1$, $C_2$, $P_1''$, and $P_2''$.
Since $P_1+P_2$ is not good, $P_1''+P_2''$ is not good either, 
and hence the induction hypothesis implies that $A=P_1''+(C_2-P_2'')$ is good.

Since $A$ is good, it contains a cycle $C_3\in\Cscr$.
If $C_3$ does not contain any edge of $P_2'$, we conclude that $P_1+(C_2-P_2)$ is good, as required.  
Otherwise (by the choice of $x$), $C_3$ contains $P_2'$ entirely.

Now we apply the induction hypothesis to $C_1$, $C_3$, $P_1'$, and $P_2'$.
Since $P_1+P_2$ is not good, $P_1'+P_2'$ is not good either, 
and hence the induction hypothesis implies that $B=P_1'+(C_3-P_2')$ is good. 
Since $B\subseteq P_1'+(A-P_2')=P_1+(C_2-P_2)$, we conclude that $P_1+(C_2-P_2)$ is good.
\end{proof}

\begin{figure}[htb]
 \begin{center}
  \begin{tikzpicture}[scale=0.6, thick]
  \node[blue, circle, fill, inner sep=0, minimum size=5] (v) at (0,0) {};
  \node[blue, circle, fill, inner sep=0, minimum size=5] (w) at (10,0) {};
  \draw[blue, very thick] (v) -- (w);
  \node[red, circle, fill, inner sep=0, minimum size=5] (x) at (5,0) {};
  \node[blue,below=0.5mm] at (v) {\small $v$};
  \node[blue,below=0.5mm] at (w) {\small $w$};
  \node[red,above=0.5mm] at (x) {\small $x$};
  \node[red] at (2.5,1.2) {\small $P_2'$};
  \node[red] at (6,-0.8) {\small $P_2''$};
  \node[blue] at (1.5,-0.4) {\small $P_1'$};
  \node[blue] at (8.5,-0.4) {\small $P_1''$};
  \draw[red] (v) to[bend left] (x);
  \draw[red] (x) .. controls (4.5,-0.3) .. (4,0) .. controls (3.5,0.4) and (3,0.2) .. (3,0) .. controls (3,-0.2) and (5,-1) .. (7,0) .. controls (8.5,0.5) .. (w);   
  \end{tikzpicture}
 \end{center}
 \caption{Illustrating the proof of Lemma~\ref{lemma:strongly_uncrossable} (Case 2).
 Here $P_1$ is the blue horizontal $v$-$w$-path, and $P_2$ is the red curved path.
 \label{fig:stronguncrossing}}
\end{figure}
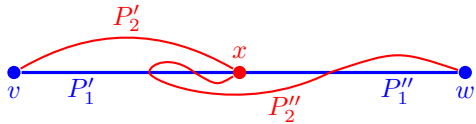

This implies:

\begin{theorem}\label{thm:strongly_uncrossable}
Any uncrossable family of cycles is strongly uncrossable.
\end{theorem}

\begin{proof}
Let $\Cscr$ be an uncrossable family of cycles and $C_1,C_2\in\Cscr$. 
Let $v$ and $w$ be two vertices that belong to both cycles $C_1$ and $C_2$, and let
$P_1$ be a $v$-$w$-path in $C_1$ and $P_2$ a $v$-$w$-path in $C_2$.
 
If $P_1+P_2$ and $(C_1-P_1)+(C_2-P_2)$ are good, we are done.
Suppose, without loss of generality, $P_1+P_2$ is not good.
Then we apply Lemma~\ref{lemma:strongly_uncrossable} (once to $C_1,C_2,P_1,P_2$, and once to $C_2,C_1,P_2,P_1$) and obtain that
$P_1+(C_2-P_2)$ is good and $P_2+(C_1-P_1)$ is good. Again, we are done. 
\end{proof}

\subsection{Uncrossing in bounded-genus graphs}

We have defined when two separating cycles cross and will now extend this to non-separating cycles and count the number of crossings.

\begin{definition}[crossings, uncrossed] \label{def:crossings}
Let $C_1$ and $C_2$ be two cycles in a graph $G$ embedded in an orientable surface.
A \emph{set of crossings} of $C_1$ and $C_2$ is a set $X$ of vertices of $G$ such that for every 
$\epsilon>0$ there are simple closed curves $\tilde C_1$ and $\tilde C_2$ in an $\epsilon$-environment of the
embeddings of $C_1$ and $C_2$, respectively, such that $\tilde C_1\cap\tilde C_2=X$ and $|X|$ is minimum.
We say that $C_1$ and $C_2$ cross $|X|$ times.

We call a multi-set of cycles \emph{uncrossed} if no pair of its cycles crosses more than once.
\end{definition}

Note that the set $X$ of crossings is not unique in general 
(if $C_1$ and $C_2$ share a path incident to a crossing, any vertex of that path could be chosen as crossing).
We remark that the definition of \cite{HuaMMV21} is slightly different but equivalent.

Definition~\ref{def:crossings} is consistent with the previous definition: a pair of separating cycles crosses 
if and only if the number of their crossings is at least one (in fact at least two). 
We note:

\begin{proposition}\label{prop:uncrossed_separating_are_laminar}
If a set of separating cycles is uncrossed, then it is laminar.
\end{proposition}

\begin{proof}
Every pair of separating cycles crosses an even number of times.
Hence, if a set of separating cycles is uncrossed, no pair of its cycles cross.
\end{proof}

However, for non-separating cycles there is no notion of laminarity.
Nevertheless we can uncross cycles in an uncrossable family if a pair of cycles crosses at least twice.
The following generalizes Proposition~\ref{prop:uncross_planar}; its proof follows closely \cite{HuaMMV21}.

\begin{lemma}\label{lemma:uncross_bounded_genus}
Let $G=(V,E)$ be a graph embedded in a fixed orientable surface, and let
$\Cscr$ be an uncrossable family in $G$ given by a weight oracle. 
Given an explicit multi-subset $\Fscr\subseteq\Cscr$,
one can compute in polynomial time an uncrossed multi-subset $\Lscr\subseteq\Cscr$ with $|\Lscr|=|\Fscr|$ and such that
for every vertex $v$ and every edge $e$, the number of cycles that contain $v$ (or $e$) is no more in $\Lscr$ than in $\Fscr$.
\end{lemma}
\begin{proof}
First, we can assume without loss of generality that all cycles in $\Fscr$ are pairwise edge-disjoint; 
this can be achieved by replacing each edge by sufficiently many parallel edges and can be done without increasing the number of crossings.
This makes the set $X$ of crossings in Definition~\ref{def:crossings} unique.
Therefore, for a vertex $v\in V$ we can define $\crossings(v)$ to be the number of pairs of cycles in $\Fscr$ that cross in $v$.

While $\Fscr$ is not already uncrossed, take $C_1, C_2 \in \Fscr$ and two vertices $v, w$ of $G$ such that $C_1$ and $C_2$ cross in $v$ and in $w$.
By Theorem~\ref{thm:strongly_uncrossable} and Definition~\ref{def:strongly_uncrossable}, there are
$v$-$w$-paths $P_1$ in $C_1$ and $P_2$ in $C_2$ such that
$P_1 + P_2$ contains a cycle $C_1^\prime\in\Cscr$ and $(C_1-P_1) + (C_2-P_2)$ contains a cycle $C_2^\prime\in\Cscr$.
We can find such cycles using Proposition~\ref{prop:membership_with_supportoracle}.
We claim that replacing $C_1$ and $C_2$ by $C_1^\prime$ and $C_2^\prime$ decreases 
$\left(\sum_{C\in\Fscr} |E(C)|, \sum_{x\in V}\crossings(x)\right)$ lexicographically. 
Since $\sum_{C\in\Fscr} |E(C)|$ can decrease at most $|\Fscr|\cdot|E|$ times
and $\sum_{x\in V}\crossings(x)$ can decrease at most $|\Fscr|^2 |V|$ times while $\sum_{C\in\Fscr} |E(C)|$ is constant,
this will complete the proof.

If at least one edge of $C_1$ or $C_2$ is neither part of $C_1^\prime$ nor or $C_2^\prime$,
then replacing $C_1$ and $C_2$ by $C_1^\prime$ and $C_2^\prime$ decreases $\sum_{C\in\Fscr} |E(C)|$.

Otherwise, $\sum_{C\in\Fscr} |E(C)|$ remains constant, 
and replacing $C_1$ and $C_2$ by $C_1^\prime$ and $C_2^\prime$ 
does not increase $\crossings(x)$ for any $x \in V \setminus \{v,w\}$.
Let now $x \in \{v,w\}$.
Since $C_1$ and $C_2$ crossed in $x$, the new cycles $C_1^\prime$ and $C_2^\prime$ do not cross in $x$.
Also, for every other cycle $C \in \Cscr \setminus \{C_1, C_2\}$ one can verify by an easy case distinction 
(cf.\ Figure~\ref{fig:uncrossing_two_cycles}) 
that it does not cross more cycles in $\{C_1^\prime,C_2^\prime\}$ at $x$ than in $\{C_1, C_2\}$.
This means that replacing $C_1$ and $C_2$ by $C_1^\prime$ and $C_2^\prime$ 
decreases the total number of crossings by at least two.

\begin{figure}[htb]
 \begin{center}
  \begin{tikzpicture}[scale=0.3, thick]
  \tikzstyle{vertex}=[circle,fill,minimum size=4,inner sep=0pt]
  \begin{scope}[shift={(0,0)}]
   \node[vertex] (x) at (0, 0) {};
   \node at (-1, -0.5) {$x$};
   \draw[red, very thick] (-6, 0) -- (x) -- (6, 0);
   \node[red] at (-6, 1) {$C_1$};
   \draw[blue, ultra thick] (0, -6) -- (x) -- (0, 6);
   \node[blue] at (1, 6) {$C_2$};
   \draw[darkgreen, very thick, dotted] (-5, -6) -- (x) -- (5, 6);
   \draw[orange, densely dotted] (-5, 6) -- (x) -- (5, -6);
   \draw[violet, densely dashed] (6, -3) -- (x) -- (6, 3);
  \end{scope}
  \begin{scope}[shift={(25, 0)}]
   \node[vertex] (x) at (0, 0) {};
   \node at (-1, -0.5) {$x$};
   \draw[red, very thick] (-6, 0) -- (x) -- (0, -6);
   \node[red] at (-6, 1) {$C_1^\prime$};
   \draw[blue, ultra thick] (0, 6) -- (x) -- (6, 0);
   \node[blue] at (1, 6) {$C_2^\prime$};
   \draw[darkgreen, very thick, dotted] (-5, -6) -- (x) -- (5, 6);
   \draw[orange, densely dotted] (-5, 6) -- (x) -- (5, -6);
   \draw[violet, densely dashed] (6, -3) -- (x) -- (6, 3);
  \end{scope}
  \end{tikzpicture}
 \end{center}
 \caption{On the left we see $C_1$ and $C_2$ crossing in $x$, together with three possible other cycles in $\Fscr$ that go through $x$ and might cross $C_1^\prime$ or $C_2^\prime$ at $x$.
 However, on the right one can see that none of the three cycles crosses more cycles in $x$ after replacing $C_1$ and $C_2$ by $C_1^\prime$ and $C_2^\prime$.\label{fig:uncrossing_two_cycles}}
\end{figure}
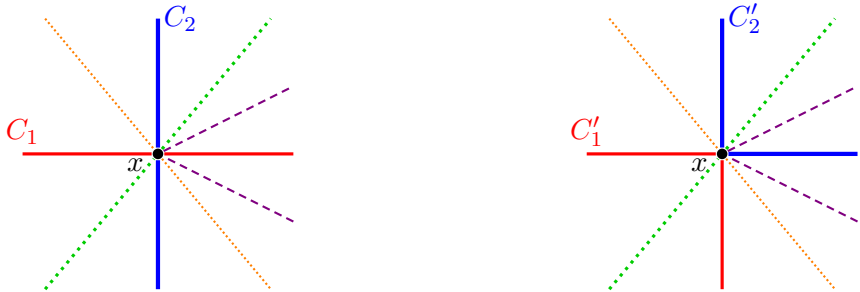

Overall, iterating those replacements yields a multi-set $\Lscr$ as demanded after at most $O(|\Fscr|^3 |V| |E|)$ iterations.
\end{proof}

\subsection{Uncrossing an LP solution}
 
We first recall the well-known fact that the linear programs~\eqref{eq:lp_edgedisjoint} and~\eqref{eq:lp} 
can be solved in polynomial time, in spite of their exponentially many variables.
In particular, the \emph{support} (the set of cycles whose variables are positive) will be polynomially bounded:

\begin{proposition}\label{prop:solve_lp}
Given a graph $G=(V,E)$ and a weight oracle for an uncrossable family $\Cscr$ of cycles in $G$, 
one can compute in polynomial time an optimum solution $x$ to the linear program~\eqref{eq:lp_edgedisjoint} 
or~\eqref{eq:lp}, given by an explicit list of all pairs $(C,x_C)$ with $C\in\Cscr$ and $x_C>0$.
\end{proposition}
 
\begin{proof} 
We focus on \eqref{eq:lp}; the proof for \eqref{eq:lp_edgedisjoint} is analogous.
First solve the dual LP
\begin{equation}\label{eq:dual_lp}
\min \left\{ \sum_{v\in V}y_v : \sum_{v\in C}y_v\ge 1 \ (C\in\Cscr),\ y_v\ge 0 \ (v\in V) \right\}
\end{equation}
by the ellipsoid method; its separation problem reduces to finding,
for given vertex weights $y_v\ge 0$ ($v\in V$), a cycle $C\in\Cscr$ for which $\sum_{v\in C}y_v$ is minimum.
This is equivalent to finding a cycle $C$ whose total edge weight is minimum, where the weight of an
edge from $v$ to $w$ is $\frac{1}{2}(y_v+y_w)$, and this is what the weight oracle does.
Finally we check whether $\sum_{v\in C}y_v < 1$. 

To solve the primal LP, we can ignore all $C\in\Cscr$ except those returned by the separation oracle while solving the dual LP.
Restricting the primal LP to these variables (whose number is bounded by a polynomial in $|V|$), we can solve it in polynomial time.
\end{proof}

\begin{lemma}\label{lemma:compute_uncrossed_lp_solution}
Let $\epsilon>0$ be a fixed constant.
Given a graph $G=(V,E)$ embedded in a fixed orientable surface and a weight oracle for an uncrossable family $\Cscr$ of cycles in $G$,
one can compute in polynomial time a feasible solution $x$ with uncrossed support 
to the linear program~\eqref{eq:lp_edgedisjoint} or~\eqref{eq:lp}, 
and such that $\sum_{C\in\Cscr} x_C$ is at least $(1-\epsilon)$ times the LP value.
\end{lemma}

\begin{proof}
First compute an optimum solution $x^*$ to the LP by applying Proposition~\ref{prop:solve_lp}.
Let $\textnormal{LP}:= \sum_{C\in\Cscr} x^*_C$ and $K:=|\{C\in\Cscr: x^*_C>0\}|$; note that $K$ is bounded by a polynomial in $|V|$.
Now define $y_C:= \lfloor\frac{K}{\epsilon\cdot\textnormal{LP}} x^*_C\rfloor$
and consider the multi-set $\Fscr$ that contains $y_C$ copies of $C$ for every $C\in\Cscr$.
Observe that $\Fscr$ contains at most $\frac{K}{\epsilon}$ cycles (counting multiplicities).
Apply Lemma~\ref{lemma:uncross_bounded_genus} to $\Fscr$ and obtain an uncrossed multi-subset $\Lscr$ of $\Cscr$.
Finally, if $\Lscr$ contains $z_C$ copies of $C$, then set $x_C:= \frac{\epsilon\cdot\textnormal{LP}}{K} z_C$  for all $C\in\Cscr$.

Obviously $\frac{\epsilon\cdot\textnormal{LP}}{K} y$ and hence $\frac{\epsilon\cdot\textnormal{LP}}{K} z=x$ is a feasible LP solution.
Moreover, 
\begin{align*}
\sum_{C\in\Cscr}x_C = \frac{\epsilon\cdot\textnormal{LP}}{K} |\Lscr| = \frac{\epsilon\cdot\textnormal{LP}}{K} |\Fscr| 
= \frac{\epsilon\cdot\textnormal{LP}}{K} \sum_{C\in\Cscr} \left\lfloor\frac{K}{\epsilon\cdot\textnormal{LP}} x^*_C \right\rfloor 
\ge \sum_{C\in\Cscr} x^*_C - \epsilon\cdot\textnormal{LP} = (1-\epsilon) \textnormal{LP}.
\end{align*}
\end{proof}

It is an open question whether an optimum LP solution with uncrossed support can be computed in polynomial time.
However, the $(1-\epsilon)$ factor will not be relevant for our main results.
As far as existence is concerned, we immediately note:

\begin{corollary}
Let $G$ be a graph and $\Cscr$ an uncrossable family $\Cscr$ of cycles in $G$.
Then the linear programs~\eqref{eq:lp_edgedisjoint} and~\eqref{eq:lp} have 
an optimum solution with uncrossed support.
\hfill\qed
\end{corollary}

For planar graphs, we can remove the $1-\epsilon$ factor, but it is not trivial:

\begin{theorem}\label{thm:compute_laminar_lp_solution}
Given a planar graph $G=(V,E)$ and a weight oracle for an uncrossable family $\Cscr$ of cycles in $G$,
one can compute in polynomial time an optimum solution $x^*$ 
to the linear program~\eqref{eq:lp_edgedisjoint} or~\eqref{eq:lp} such that the support of $x^*$ is laminar. 
\end{theorem}

\begin{proof}
Again we argue for \eqref{eq:lp}; the proof for \eqref{eq:lp_edgedisjoint} is completely analogous.
To allow for efficient uncrossing, 
we want to minimize $\sum_{C\in\Cscr}x_C |C|$ among all optimum solutions to \eqref{eq:lp}.
This can be done by first computing the optimum value $\text{OPT}$ of \eqref{eq:lp} and then solving the LP
\begin{equation}\label{eq:lp_minimize_cycle_length}
 \min \left\{ \sum_{C \in \Cscr} x_C |C| : \sum_{C \in \Cscr : v \in \Cscr} x_C \leq 1 \ (v \in V),\ \sum_{C \in \Cscr} x_C = \text{OPT},\ x_C \geq 0 \ (C\in \Cscr) \right\}.
\end{equation}

Similarly to the proof of Proposition~\ref{prop:solve_lp}, we do this by first solving the dual LP
\begin{equation}\label{eq:dual_lp_minimize_cycle_length}
\max \left\{ z \ \text{OPT} - \sum_{v\in V}y_v : z \leq \sum_{v\in C} (y_v + 1) \ (C\in\Cscr),\ y_v\ge 0 \ (v\in V) \right\},
\end{equation}
whose separation problem again reduces to calling the weight oracle, by the ellipsoid method, and keeping only primal variables corresponding to cycles returned by the separation oracle.

We now describe how we uncross the support of an optimum solution $x$ to \eqref{eq:lp_minimize_cycle_length} (which is obviously also an optimum solution to \eqref{eq:lp}).
First we describe an uncrossing operation of a pair of cycles, $C_1$ and $C_2$, that cross.
Take two vertices $v, w$ of $G$ such that $C_1$ and $C_2$ cross in $v$ and in $w$.
By Theorem~\ref{thm:strongly_uncrossable} and Definition~\ref{def:strongly_uncrossable}, 
there are $v$-$w$-paths $P_1$ in $C_1$ and $P_2$ in $C_2$ such that
$P_1 + P_2$ contains a cycle $C_1^\prime\in\Cscr$ and $(C_1-P_1) + (C_2-P_2)$ contains a cycle $C_2^\prime\in\Cscr$.
Again we can find such cycles using Proposition~\ref{prop:membership_with_supportoracle}.
If $C_1^\prime$ and $C_2^\prime$ still cross (we know they cross less often than $C_1$ and $C_2$), 
we apply the same step again to this pair of cycles.
Let $\bar C_1$ and $\bar C_2$ denote the final outcome. These two cycles do not cross. 
Now we would set $\delta:=\min\{x_{C_1},x_{C_2}\}$ and change the LP solution by subtracting $\delta$
from $x_{C_1}$ and $x_{C_2}$ and adding $\delta$ to $x_{\bar C_1}$ and $x_{\bar C_2}$.
Note that $x$ remains an optimum LP solution.
One of $C_1$ and $C_2$ vanishes from the support of $x$, but in general not both.

This operation cannot decrease $\sum_{C\in\Cscr} x_C|C|$ because $\sum_{C\in\Cscr} x_C$ remains
constant and $x$ was an optimum solution to \eqref{eq:lp_minimize_cycle_length}.
Hence every edge is contained the same number of times in $\bar C_1$ and $\bar C_2$ as in $C_1$ and $C_2$.
This is a key property that we will exploit now.

This single step makes some progress, but we need to be very careful to obtain a polynomial bound
on the number of steps.  
Fix an embedding of $G$ in the sphere and a point $\infty$ in one of the faces.
Let $\faces(C)$ again denote the set of faces in the interior of a cycle $C$.
Since $\bar C_1$ and $\bar C_2$ do not cross, $\faces(\bar C_1)$ and $\faces(\bar C_2)$
are either disjoint or one set is a subset of the other.
In the first case, $\{\faces(\bar C_1),\faces(\bar C_2)\}=\{\faces(C_1)\setminus\faces(C_2),\faces(C_2)\setminus\faces(C_1)\}$.
In the second case, $\{\faces(\bar C_1),\faces(\bar C_2)\}=\{\faces(C_1)\cap\faces(C_2),\faces(C_1)\cup\faces(C_2)\}$.

Therefore we can apply Karzanov's uncrossing algorithm \cite{Kar96} for uncrossing set systems and terminate 
with laminar support after polynomially many uncrossing steps. 
\end{proof}

\section{Packing cycles by rounding an LP solution}\label{sec:lpbased}

\subsection{Edge-disjoint packing in planar graphs} \label{section:gks}

The proof of Theorem~\ref{thm:main_lpbased_planar}\,(b) consists of first applying
Theorem~\ref{thm:compute_laminar_lp_solution} to obtain an optimum solution $x^*$ 
to the LP~\eqref{eq:lp_edgedisjoint} such that the support $\Cscr^{>0}$ of $x^*$ is laminar,
and then following the rounding algorithm of Garg, Kumar and Seb\H{o} \cite{GarKS22} without any change. 
Their algorithm only needs the explicit list $\Cscr^{>0}$; only these cycles will be considered in the following. 
First, they observe that for any edge $e$ the cycles containing $e$ form two chains, say $L_1(e)$ and $L_2(e)$, in the partial order $\subseteq_{\infty}$.
(Note that the analogue for vertex-disjoint packing is obviously false: a vertex can belong to many cycles with disjoint interior.)
This ensures that the LP
\begin{equation}\label{eq:lp_rounding_gks}
\max \left\{ \sum_{C\in\Cscr^{>0}}x_C : \sum_{C\in L_i(e)}x_C\le 1 \ (e \in E, i=1,2),\ x_C\ge 0 \ (C\in\Cscr) \right\}
\end{equation}
has an integral optimum solution, computable in polynomial time. This solution is given by a subset $\Cscr_{1/2}\subseteq\Cscr^{>0}$, such that setting
$x_C=\frac{1}{2}$ for all $C\in\Cscr_{1/2}$ and $x_C=0$ for all other cycles $C$ constitutes
a feasible LP solution with $|\Cscr_{1/2}|\ge\sum_{C\in\Cscr}x^*_C$.
Then they exploit an observation of \cite{FioXX07} that the conflict graph (with vertex set $\Cscr_{1/2}$
and edges between cycles that share an edge) is planar. 
Therefore, using (an algorithmic version of) the four-color theorem one can find a subset 
$\Cscr_1\subseteq\Cscr_{1/2}$ of pairwise edge-disjoint cycles with $|\Cscr_1|\ge \frac{1}{4}|\Cscr_{1/2}|$.
This proof does not use any specific properties of $D$-cycles once we have the LP solution with laminar support.
Hence Theorem~\ref{thm:main_lpbased_planar}\,(b) follows.

\subsection{Vertex-disjoint packing: overview of our approach}

The core of the proofs of Theorem~\ref{thm:main_lpbased_planar}\,(a) 
and Theorem~\ref{thm:main_lpbased_boundedgenus} is the same. 
Let us first concentrate on Theorem~\ref{thm:main_lpbased_planar}\,(a), 
which deals with vertex-disjoint cycle packing in planar graphs.
Then all cycles are separating cycles.

On a high level, our algorithm to prove Theorem~\ref{thm:main_lpbased_planar}\,(a) consists of two steps.
Assume that $G$ is embedded in the sphere.
First we find a solution to the linear program \eqref{eq:lp} and apply uncrossing to obtain a solution with laminar support;
i.e., we apply Theorem~\ref{thm:compute_laminar_lp_solution}.
Let $\Cscr^{>0}=\{C\in\Cscr: x_C>0\}$ denote the support of our final LP solution $x$;
the set $\Cscr^{>0}$ is laminar.
From now on we will only work with $\Cscr^{>0}$.

Now the main (and fundamentally new) part of our algorithm begins.
It consists of a greedy algorithm that is guided by this LP solution.
We find a cycle $C^*\in\Cscr^{>0}$ and a set $W$ of at most five vertices such that
every cycle $C\in\Cscr^{>0}$ that shares a vertex with $C^*$ contains a vertex from $W$.
This is always possible due to our following key lemma (applied to $\Lscr=\Cscr^{>0}$):

\begin{lemma}[Efficient Cycle Lemma] \label{lemma:efficient_cycle_lemma_planar}
Let $G$ be a planar graph embedded in the sphere, and let $\Lscr$ be a non-empty laminar set of cycles in $G$. 
Then there exists a cycle $C^* \in \Lscr$ and a set $W$ of its vertices
such that $|W|\leq 5$ and every $C \in \Lscr$ is either vertex-disjoint from $C^*$ or contains at least one vertex from $W$.
\end{lemma}

We will prove the Efficient Cycle Lemma in the following sections.
Given this, we can complete the proof of Theorem~\ref{thm:main_lpbased_planar}\,(a) easily.
Let $C^*$ be a cycle and $W$ a vertex set as guaranteed by the Efficient Cycle Lemma
(they can be found in polynomial time by complete enumeration).  
We include $C^*$ in our solution and reset $x_C=0$ 
for all cycles $C$ that share a vertex with $C^*$. 
We get a feasible LP solution for the cycle packing instance on $G-V(C^*)$, 
whose value is smaller by at most 5 because
\begin{align*}
\sum_{C\in\Cscr:V(C)\cap V(C^*)\not=\emptyset} x_C \, = \sum_{C\in\Cscr: V(C)\cap W\not=\emptyset} x_C
\, \le \sum_{w\in W} \sum_{C\in\Cscr: w\in V(C)} x_C \, \le \, |W| \, \le \, 5.
\end{align*}
Iterating this process completes the proof of Theorem~\ref{thm:main_lpbased_planar}\,(a).

In Section~\ref{section:packing_cycles_bounded_genus} we explain the differences in the bounded-genus case.
Here we obtain an LP solution with uncrossed support, and we proceed exactly as above for the set of separating cycles.
The only difference is that the constant 5 in Lemma~\ref{lemma:efficient_cycle_lemma_planar} 
will depend on the genus.
To complete the proof of Theorem~\ref{thm:main_lpbased_boundedgenus}\,(a),
we follow \cite{HuaMMV21} if most of the LP value is concentrated on non-separating cycles; 
see Section~\ref{section:packing_cycles_bounded_genus}.

Finally, we will show how to obtain the edge-disjoint version (b) of Theorem~\ref{thm:main_lpbased_boundedgenus} 
in Section~\ref{section:lprounding_edgedisjoint}.
 
 \subsection{Nice paths} 
 
 In the next three sections we will prove Lemma~\ref{lemma:efficient_cycle_lemma_planar}. 
 In fact, we immediately prove the generalization for the bounded-genus case; the only
 difference will be that the constant 5 will depend on the genus.
So let $G$ be a graph embedded in an orientable surface of genus $g$, the sphere if $G$ is planar. 
 Let also $\Lscr$ be a laminar family of separating cycles in $G$. 
 We assume $|\Lscr| \geq 2$ since otherwise Lemma~\ref{lemma:efficient_cycle_lemma_planar} is trivial. 
 This implies that for each cycle in $\Lscr$ one of the sides strictly contains a side of another cycle.
 
 \begin{definition}[one-sided, two-sided, $\subseteq_C$] \label{def:one-sided_two-sided}
 In the above situation, if both sides of a cycle $C \in \Lscr$ contain a side of another cycle of $\Lscr$, then $C$ is called \emph{two-sided}.
 Otherwise, $C$ is called \emph{one-sided} and we denote by $S(C)$ the side that contains no side of another cycle.
 Furthermore, for a one-sided cycle $C$ we define a partial order $\subseteq_C$ on $\Lscr$ by $C_1 \subseteq_C C_2$ 
 if and only if a side of $C_1$ is contained in the (unique) side of $C_2$ that does not contain $S(C)$.
\end{definition}

Note that face-minimal cycles are one-sided, but there can be one more one-sided cycle (whose interior must then contain all other cycles in $\Lscr$).

We will see that it is always possible to choose a one-sided cycle $C^*$ in Lemma~\ref{lemma:efficient_cycle_lemma_planar}. 
We will construct a graph $G^\prime$ and embed it in the same surface as $G$.
The vertices of $G^\prime$ are the one-sided cycles, and the edges represent conflicts. 
In the end, $C^*$ will correspond to a low-degree vertex in $G^\prime$.
Since many one-sided cycles can share a vertex, our graph will not contain
an edge for all pairs of such cycles. Moreover, we have to take conflicts to two-sided cycles into account.
To embed the edges of $G^\prime$, we will use \emph{nice paths}:

\begin{definition}[nice path] \label{def:nice_path}
 Let $S$ be a side of a cycle $C \in \Lscr$ and $x$ a point on the embedding of $C$  that does not lie on the embedding of any other cycle contained in $S$.
 A \emph{nice path} for $(x, S)$ is a continuous path $P$ on the orientable surface that $G$ is embedded in such that
 \begin{enumerate}
  \item $P$ starts in $x$ and ends in a point on the embedding of some one-sided cycle $C'$ such that $S(C')$ is contained in $S$
  \item Except for the start point $x$, $P$ is contained in $S$
  \item $P$ intersects any cycle in $\Lscr$ in at most one point
 \end{enumerate}
\end{definition}

\begin{remark}\label{remark:concatenation_nice_paths}
 If $(x, S)$ and $(x, S^\prime)$ are two pairs as in the above definition such that $S$ and $S^\prime$ are disjoint (or equivalently $S \neq S^\prime$), it is easy to check that the concatenation of nice paths for $(x, S)$ and $(x, S^\prime)$ is again a nice path in both directions.
\end{remark}

\begin{lemma}\label{lemma:nice_paths}
 Let $T$ be a finite set of pairs $(x, S)$ as in Definition~\ref{def:nice_path}. Then there are nice paths $(P_t)_{t \in T}$ for the $t \in T$ such that all the $P_t$ are pairwise disjoint, except for possibly coinciding start points $(x, S), (x, S^\prime) \in T$.
\end{lemma}
\begin{proof}
Let $X$ be the set of all start points, i.e., $X$ contains $x$ for all $(x,S)\in T$ and no other points.  
We construct the paths one by one, ensuring in addition that the nice path $P_t$ for $t=(x,S)\in T$ contains no point in $X\setminus\{x\}$.
Assume that there are already nice paths $P_t$ with the desired properties for all $t\in T' \subset T$, and let $t_0\in T\setminus T'$. 
We show that there is a nice path for $t_0 = (x_0, S_0) \in T$ that (except possibly at $x_0$) avoids $X$ and all previously constructed paths. 

We prove this by induction on the number of cycles of $\Lscr$ that have a side strictly contained in $S_0$.
If $S_0$ does not strictly contain a side of any cycle in $\Lscr$, 
$x_0$ is already a feasible endpoint and we can take a trivial path as $P_{t_0}$.
Otherwise, consider the connected components that arise from $S_0$ after deleting 
the embeddings of the cycles in $\Lscr$ and the previously constructed paths $P_t$ ($t\in T'$). 
Let $A$ be a connected component such that $x_0$ is on the boundary of $A$.

\begin{figure}[htb]
 \begin{center}
  \begin{tikzpicture}[scale=0.4, thick]
  \tikzstyle{vertex}=[red,circle,fill,minimum size=4,inner sep=0pt]
  \node[vertex] (x) at (-3.74, -4.14) {};
  \draw[white, fill=blue!7] (x) to[out=225, in=150] (-2, -6.73) arc (154:270:2) -- (-0.2, -9.6) arc (270:90:9 and 7) arc (90:225:5);
  \begin{scope}[shift={(-1.6,-0.6)}]
  \draw[fill=darkgreen!5] (0, 0) circle (3.6);
  \draw (0, 0) circle (1.6);
  \draw (120:2.6) circle (1);
  \draw (-90:2.6) circle (1);
  \end{scope}
  \draw (5.8, -4.7) circle (2.22);
  \draw (-0.2, -0.6) circle (5);
  \draw (-0.2, -7.6) circle (2);
  \draw (-0.2, -2.6) ellipse (9 and 7);
  \draw (0, -7.5) circle (0.9);
  \node[vertex,red] at (x) {};
  \node[vertex,blue] (x0) at (-9.2, -2.6) {};
  \node[vertex, darkgreen] (y) at (-5.2, -0.6) {};
  \draw[brown, densely dashed] (x) to (210:3.6) to (-3,-1.3);
  \draw[red, densely dashed] (x) to[out=225, in=150] (-2, -6.73) to[out=-30] (-0.9, -7.5);
  \draw[blue, densely dotted] (x0) to[in=210] (y);
  \draw[darkgreen, densely dotted] (y) to (-3.8,1.2);
  \node[blue, anchor=east] at (x0) {$x_0$};
  \node[red, anchor=west] at (-4.3,-4.8) {$x$};
  \node[darkgreen] at (-4.7, -1.1) {$y$};
  \node[darkgreen] at (1, -1.5) {$S'$};
  \node[blue] at (-6.5, -4) {$A$};
  \draw[->] (x0) -- +(0:0.8);
  \draw[->] (x) -- +(0:0.8);
  \draw[->] (x) -- +(180:0.8);
  \draw[->] (y) -- +(20:0.8);
  \end{tikzpicture}
 \end{center}
 \caption{Example for nice paths. The set $X$ contains two points $x$ and $x_0$. For the point $x$ nice paths to both incident sides are demanded and have already been found (dashed). 
Note that the concatenation is again a nice path in both directions. 
Now we ask for a nice path for $(x_0,S_0)$, where $S_0$ is the interior of the largest cycle in the figure.
$x_0$ can be connected feasibly to all points on the boundary of the blue area $A$. 
We choose a point $y$ on a cycle with interior $S'$ inside $S_0$ 
(and choose that cycle so that $S'$ is minimal) and complete an $x_0$-$y$-path inside $A$ (blue, dotted) 
to a nice path by a nice path for $(y,S')$ (green, dotted), which exists by induction.
\label{fig:nice_paths_lemma}}
\end{figure}
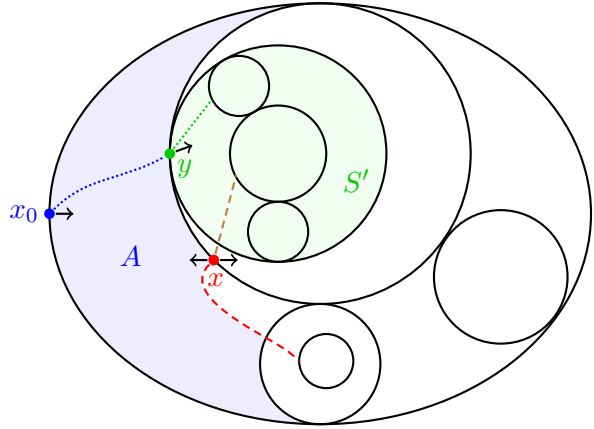

If two previously constructed paths have a common point, their concatenation is also a nice path by Remark~\ref{remark:concatenation_nice_paths}.
Any nice path $P_t$ can touch the boundary of $S_0$ at most once.
Moreover $S_0$ strictly contains a side of a cycle in $\Lscr$.
Hence the boundary of $A$ contains a point $y$ that is neither on the boundary of $S_0$ 
nor on a previously constructed path $P_t$ ($t\in T'$) nor in $X$ (cf.~Figure~\ref{fig:nice_paths_lemma}).
 
Therefore $y$ must lie on a cycle $C \in \Lscr$ with a side $S^\prime$ strictly contained in $S_0$, such that $S^\prime$ does not contain any other cycle that touches $y$.
Then the induction hypothesis implies that there is a nice path $P^\prime$ for $(y, S^\prime)$ 
that is disjoint to each of the $P_t$ ($t\in T'$) and to $X$.
By prepending an $x_0$-$y$-path whose interior is inside $A$ we get the desired nice path for $t_0$.
\end{proof}

\subsection{Homotopic edges and low-degree vertices}
 
As already mentioned, we will prove the Efficient Cycle Lemma~\ref{lemma:efficient_cycle_lemma_planar} 
and its generalization for the bounded-genus case by constructing another graph 
that can be embedded in the same orientable surface of genus $g$ and then finding a vertex of small degree in that graph. 
Since the graph can have parallel edges (but no loops), we need the notion of homotopic edges before we can apply Euler's formula:

\begin{definition}[homotopic edges] 
Let $G$ be a (directed or undirected) graph embedded in an orientable surface of genus $g$. 
Two parallel edges $e$ and $e^\prime$ of $G$ are called \emph{homotopic} 
if the embedding of $e$ and $e^\prime$ bounds an area homeomorphic to the disk without any vertex of $G$ inside.
\end{definition}

\begin{lemma}\label{lemma:mindegree_undirected}
Every undirected graph $G$ that can be embedded in an orientable surface of genus $g$ without a pair of homotopic edges contains a vertex of degree at most $6g+5$.
\end{lemma}

\begin{proof}
First, add edges to $G$ until the embedding is cellular, preserving the property that no pair of edges is homotopic. 
This cannot decrease the minimum vertex degree.
Let $V,E,F$ denote the sets of vertices, edges, faces in this graph with its cellular embedding.

Now, by Euler's formula for cellular embeddings we have $|V|-|E|+|F|=2-2g$. 
Since there are no homotopic edges, every face is bounded by at least three edges, which implies $2|E|\geq 3|F|$. 
Together this yields $|E|\leq 3|V|+6g-6$, and the minimum degree is at most
\[ \left\lfloor \frac{2|E|}{|V|}\right\rfloor\leq\left\lfloor\frac{6|V|+12g-12}{|V|}\right\rfloor\leq 6+\left\lfloor\frac{12}{|V|}(g-1)\right\rfloor\leq 6g+5\] as claimed.
\end{proof}

In the proof of the Efficient Cycle Lemma we will need a directed version of the above lemma, which follows immediately:

\begin{lemma}\label{lemma:degree5}
Every digraph that can be embedded in an orientable surface of genus $g$ without a pair of homotopic edges contains a vertex of out-degree at most $6 g + 5$.
\end{lemma}
\begin{proof}
 Ignoring the orientation of all edges can only make pairs of edges homotopic that were oriented in opposite directions; 
 therefore merging such pairs and applying Lemma~\ref{lemma:mindegree_undirected} proves the assertion. 
\end{proof}

\subsection{Proof of the Efficient Cycle Lemma} \label{section:proofofmainlemma}

Now we can prove Lemma~\ref{lemma:efficient_cycle_lemma_planar}, and its generalization to bounded-genus graphs. 
We re-state it here in a slightly stronger way by enforcing $C^*$ to be one-sided.

\begin{lemma}\label{lemma:efficient_cycle_lemma_general}
Let $G$ be a graph embedded in an orientable surface of genus $g$, the sphere if $G$ is planar. Let $\Lscr$ be a non-empty laminar set of separating cycles in $G$.
Then there exists a one-sided cycle $C^* \in \Lscr$ and a set $W$ of its vertices such that $|W| \leq 6 g + 5$ and every $C \in \Lscr$ is either vertex-disjoint from $C^*$ or contains at least one vertex from $W$.
\end{lemma}

\begin{proof}
We may assume  that $G$ is connected and every edge belongs to a cycle in $\Lscr$ (otherwise we delete redundant edges and consider connected components separately). Moreover, we may assume $|\Lscr|\ge 2$. 
We call two cycles in $\Lscr$ \emph{neighbours} if they share at least one vertex.

For a one-sided cycle $C \in \Lscr$ let $\Nscr(C)$ be the set of $\subseteq_C$-minimal neighbours in $\Lscr$.
For $N \in \Nscr(C)$ let $W(C, N)$ be the set of vertices $w \in C \cap N$ 
such that the edge in $C$ that leaves $w$ in anti-clockwise direction (with respect to the side $S(C)$) is not in $N$
(cf.~Figure~\ref{fig:define_W_sets}).
Let $W(C)$ be a minimal set such that $W(C) \cap W(C,N) \neq \emptyset$ for all $N \in \Nscr(C)$.
Now it suffices to find a one-sided cycle $C$ such that $|W(C)| \leq 6 g + 5$:
Let $N$ be some cycle that shares a vertex with $C$.
Pick $N^\prime \in \Nscr(C)$ such that $N^\prime \subseteq_C N$.
Now there is some $w \in N^\prime \cap W(C)$ and since $N^\prime \subseteq_C N$, we also have $w \in N \cap C$.

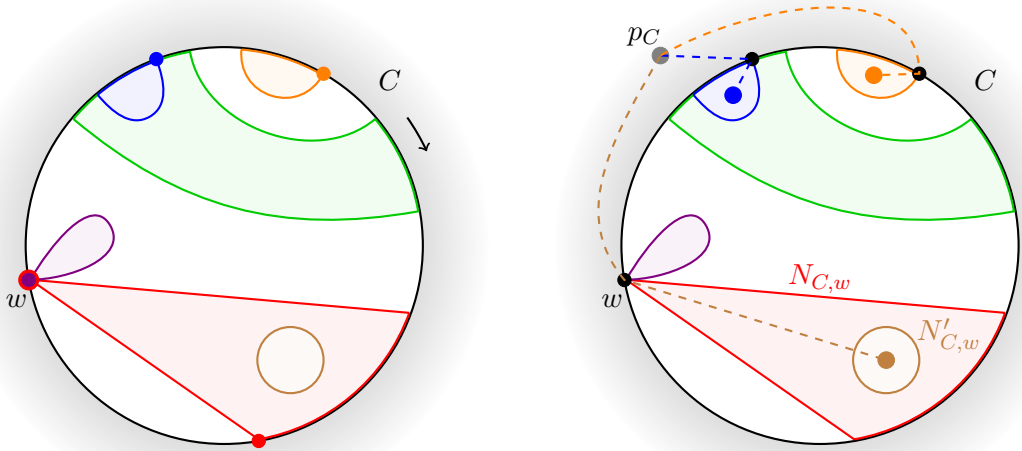
\begin{figure}
 \begin{center}
  \begin{tikzpicture}[scale=0.22]
   \begin{scope}[shift={(0,0)}]
   \filldraw[even odd rule, color=white, inner color=black!50, outer color=white]
   (0, 0) circle (12)
   (0, 0) circle (16);
   \draw[thick] (0,0) circle (12);
   \draw[->, thick] (35:13.5) arc (35:25:13.5);
   \draw[thick, red, fill=red!5] (-20:11.9) arc (-20:-80:11.9) to (-170:11.9) to (-20:11.9);
   \draw[thick, violet, fill=violet!5] (-170:11.9) to[out=0, in=-60] (170:7) to[out=120, in=60] (-170:11.9);
   \draw[thick, darkgreen, fill=darkgreen!5] (140:11.9) arc(140:100:11.9) to[out=-80, in=-140] (40:11.9) arc (40:10:11.9) to[out=-170, in=-40] (140:11.9);
   \draw[thick, blue, fill=blue!5] (130:11.9) arc (130:110:11.9) to[out=-70, in=30] (120:9) to[out=-150, in=-50] (130:11.9);
   \draw[thick, orange, fill=orange!5] (85:11.9) arc (85:60:11.9) to[out=240, in=-17.5] (72.5:9.5) to[out=162.5, in=265] (85:11.9);
   \draw[thick, brown, fill=brown!5] (-60:8) circle (2);
   \draw[fill, red] (190:12) circle (0.6);
   \draw[fill, orange] (60:12) circle (0.4);
   \draw[fill, blue] (110:12) circle (0.4);
   \draw[fill, violet] (190:12) circle (0.4);
   \draw[fill, red] (-80:12) circle (0.4);
   \node at (10, 10) {$C$};
   \node at (195:13) {$w$};
   \end{scope}
   \begin{scope}[shift={(36,0)}]
   \filldraw[even odd rule, color=white, inner color=black!50, outer color=white]
   (0, 0) circle (12)
   (0, 0) circle (16);
   \draw[thick] (0,0) circle (12);
   \draw[thick, red, fill=red!5] (-20:11.9) arc (-20:-80:11.9) to (-170:11.9) to (-20:11.9);
   \draw[thick, violet, fill=violet!5] (-170:11.9) to[out=0, in=-60] (170:7) to[out=120, in=60] (-170:11.9);
   \draw[thick, darkgreen, fill=darkgreen!5] (140:11.9) arc(140:100:11.9) to[out=-80, in=-140] (40:11.9) arc (40:10:11.9) to[out=-170, in=-40] (140:11.9);
   \draw[thick, blue, fill=blue!5] (130:11.9) arc (130:110:11.9) to[out=-70, in=30] (120:9) to[out=-150, in=-50] (130:11.9);
   \draw[thick, orange, fill=orange!5] (85:11.9) arc (85:60:11.9) to[out=240, in=-17.5] (72.5:9.5) to[out=162.5, in=265] (85:11.9);
   \draw[fill] (60:12) circle (0.4);
   \draw[fill] (110:12) circle (0.4);
   \draw[fill] (190:12) circle (0.4);
   \draw[fill, gray] (130:15) circle (0.5);
   \draw[fill, blue] (120:10.5) circle (0.5);
   \draw[fill, orange] (72.5:10.8) circle (0.5);
   \draw[thick, brown, fill=brown!5] (-60:8) circle (2);
   \draw[fill, brown] (-60:8) circle (0.5);
   \node at (10, 10) {$C$};
   \node at (195:13) {$w$};
   \node at (130:16.5) {$p_C$};
   \draw[dashed, thick, blue] (130:15) to (110:12) to (120:10.5);
   \draw[dashed, thick, orange] (130:15) to[out=30, in=90] (60:12) to (72.5:10.8);
   \draw[dashed, thick, brown] (130:15) to[out=-120, in=130] (190:12) to (-60:8);
   \node[red] at (0, -2) {$N_{C,w}$};
   \node[brown] at (-35:9.5) {$N^\prime_{C,w}$};
   \end{scope}
  \end{tikzpicture}
 \end{center}
 \caption{The left picture shows a one-sided cycle $C$ and all cycles that share a vertex with $C$. 
 Note that $S(C)$ is drawn as the outer face in this example; therefore traversing $C$ in anti-clockwise direction with respect to the side $S(C)$ corresponds to traversing $C$ in clockwise direction in this image, as indicated by the arrow. 
 The set $\Nscr(C)$ contains all the colored cycles except the brown (which is not a neighbour) and the green one (which is not a $\subseteq_C$-minimal neighbour). 
 For each $N \in \Nscr(C)$ the set $W(C,N)$ is marked by points in the same color; the specified vertex $w$ is contained in $W(C,N)$ for both the red and the violet cycle. \\
 In the right picture the black vertices on $C$ mark the (here unique) set $W(C)$. 
 The dashed lines represent embeddings of the edges that are added for $C$ to $G^\prime$. 
 Note that the red cycle is two-sided and the corresponding edge ends inside a one-sided cycle $N^\prime_{C,w}$ inside $N_{C,w}$.
 \label{fig:define_W_sets}}
\end{figure}

Let $\Lscr_1$ be the set of one-sided cycles in $\Lscr$. 
In order to find a cycle $C\in\Lscr_1$ with $|W(C)| \leq 6 g + 5$, 
we first construct an embedding of a directed graph $G^\prime$ with vertex set $\Lscr_1$ in the given surface of genus $g$.
For each $C \in \Lscr_1$ we choose a point $p_C$ in $S(C)$.
Now for each $C \in \Lscr_1$ and $w \in W(C)$ let $N_{C,w}$ be the cycle in $\Nscr(C)$ 
that comes in anti-clockwise direction first after $C$ in $w$. 
We can find a nice path for $(w, S)$ from $w$ to the boundary of some one-sided cycle $N_{C,w}^\prime$, 
where $S$ is the side of $N_{C,w}$ that does not contain $S(C)$. 
We add the edge $(C, N_{C,w}^\prime)$ to $G^\prime$ and get an embedding of that edge by 
prepending a $p_C$-$w$-path inside $S(C)$ and appending a path to $p_{N_{C,w}^\prime}$ inside $S(N_{C,w}^\prime)$ to the nice path.

By Lemma~\ref{lemma:nice_paths} the nice paths can be drawn disjointly (except for coinciding start points), 
therefore the embeddings of the edges in $G^\prime$ can intersect only in their endpoints and in points $w \in W(C) \cap W(C^\prime)$. 
However, in this case the path that was added for $(C, w)$ is continued inside a cycle that comes in anti-clockwise direction before $C^\prime$ in $w$, 
while the path that was added for $(C^\prime, w)$ is continued inside a cycle that comes in anti-clockwise direction before $C$ in $w$
(cf.\ Figure~\ref{fig:nocrossing_at_w}). 
Therefore the embeddings of edges in $G^\prime$ can only 'touch' but not 'cross'; in particular they induce an embedding of $G^\prime$.

\begin{figure}[htb]
 \begin{center}
  \begin{tikzpicture}[scale=0.6, very thick]
  \begin{scope}[shift={(180:0.05)}]
   \fill[draw=white, left color=white, right color=darkgreen!10] (-7, 3.5) to[out=0, in=145] (145:4) -- (0, 0) -- (-160:4) to[out=-160, in=0] (-7, -2) -- (-7, 3.5);
   \draw[darkgreen] (-7, 3.5) to[out=0, in=145] (145:4) -- (0, 0) -- (-160:4) to[out=-160, in=0] (-7, -2);
  \end{scope}
  \begin{scope}[shift={(-95:0.05)}]
   \draw[blue, fill=blue!5] (0, 0) -- (-30:4) to[out=-30, in=0] (3, -5) -- (-3.5, -5) to[out=180, in=-160] (-160:4) -- (0, 0);
  \end{scope}
  \begin{scope}[shift={(25:0.05)}]
   \fill[draw=white, left color=red!10, right color=white] (7, -3) to[out=180, in=-30] (-30:4) -- (0, 0) -- (80:4) to[out=80, in=180] (7, 6) -- (7, -3);
   \draw[red] (7, -3) to[out=180, in=-30] (-30:4) -- (0, 0) -- (80:4) to[out=80, in=180] (7, 6);
  \end{scope}
  \fill[draw=white, left color=violet!10, right color=white] (7, -1.5) -- (0, 0) -- (7, 2.5) -- (7, -1.5);
  \draw[violet] (7, -1.5) -- (0, 0) -- (7, 2.5);
  \draw[orange, fill=orange!5] (0, 0) to[out=35, in=-40] (50:6.5) to[out=140, in=65] (0, 0);
  \draw[gray, fill=gray!5] (0, 0) to[out=90, in=25] (115:6.5) to[out=205, in=140] (0, 0);
  
   \draw[blue, fill] (-105:3) circle (0.1);
   \draw[orange, fill] (52:4.3) circle (0.1);
   \draw[gray, fill] (120:4.5) circle (0.1);
   \node[blue, below] at (-105:3) {$p_{C_1}$};
   \draw[blue, densely dotted] (-105:3) -- (0, 0) -- (7, 0.5);
   \node[orange, above] at (52:4.3) {$p_{C_2}$};
   \draw[orange, densely dotted] (52:4.3) -- (0, 0) .. controls (110:2.25) .. (120:4.5);
   \node[gray, above] at (120:4.5) {$p_{C_3}$};
   \draw[gray, densely dotted] (120:4.5) .. controls (130:2.25) .. (0, 0) -- (-7, 0.5);
  \draw[black, fill] (0, 0) circle (0.1);
  \node[black] at (0.2, -0.55) {$w$};
  \node[gray] at (108:5.7) {$C_3$};
  \node[orange] at (47:5.9) {$C_2$};
  \node[violet] at (7, -1.1) {$N_2$};
  \node[red] at (7, -2.6) {$N_1$};
  \node[blue] at (-3.5, -4.5) {$C_1$};
  \node[darkgreen] at (-7, 3) {$N_3$};
  \end{tikzpicture}
 \end{center}
 \caption{Three one-sided cycles, $C_1$ to $C_3$, and three two-sided cycles, $N_1$ to $N_3$, meet in the point $w$. 
 Note that $w$ can be in $W(C_i)$ for all $1 \leq i \leq 3$ here although e.g. $w \notin W(C_1, N_3)$. 
 The dotted lines represent the constructed embeddings of edges added to $G^\prime$.
 \label{fig:nocrossing_at_w}}
\end{figure}
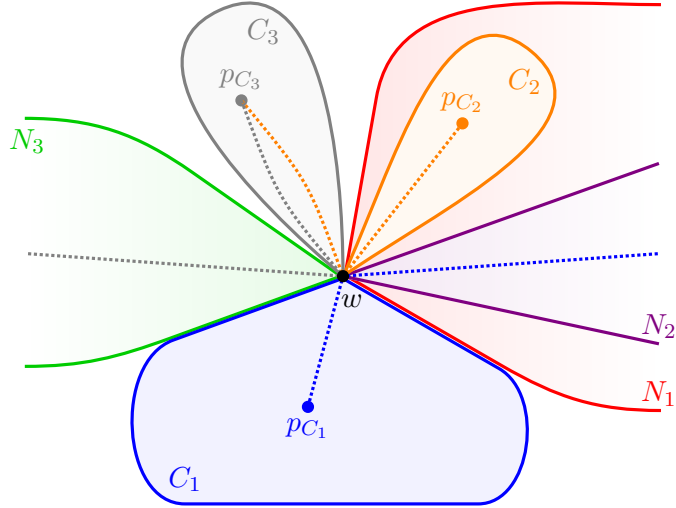

Note that the out-degree of any vertex $C\in\Lscr_1$ in this digraph $G^\prime$ is exactly $|W(C)|$.
We claim that $G^\prime$ contains no pair of homotopic edges.
To prove this, assume there are $C_1, C_2 \in \Lscr_1$ such that 
homotopic edges from $C_1$ to $C_2$ were added for $(C_1, w)$ and $(C_1, w^\prime)$.
Since both edges point to $p_{C_2}$, we know that $N_{C_1, w}$ and $N_{C_1, w^\prime}$ must be ordered by $\subseteq_{C_1}$.
Furthermore, both $N_{C_1, w}$ and $N_{C_1, w^\prime}$ are in $\Nscr(C_1)$ and thus $N_{C_1, w} = N_{C_1, w^\prime}$ (cf.\ Figure~\ref{fig:no_homotopic}).
However, due to the minimality of $W(C)$ there must be other $\subseteq_{C_1}$-minimal cycles 
$C_3$ at $w$ and $C_4$ at $w'$ that come after $N_{C_1, w}$ in anti-clockwise order (cf.\ Figure~\ref{fig:no_homotopic}).
The cycles $C_3$ and $C_4$ are not necessarily one-sided, but both sides contain a one-sided cycle
and hence a vertex of $G'$. 

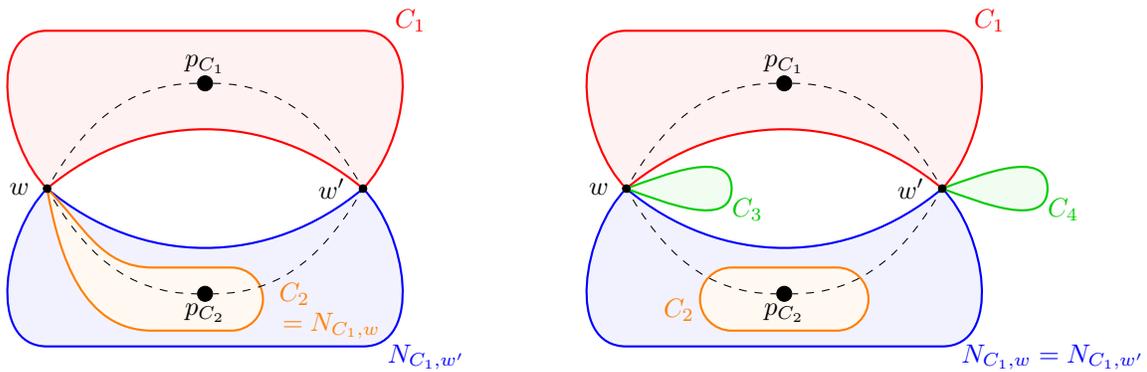
\begin{figure}[htb]
 \begin{center}
  \begin{tikzpicture}[scale=0.7]
    \begin{scope}[shift={(0,0)}]
   \draw[thick, red, fill=red!5] (0, 0) to[out=40, in=140] (6, 0) to[out=45, in=0] (6, 3) -- (0, 3) to[out=180, in=135] (0, 0);
   \draw[thick, blue, fill=blue!5] (0, 0) to[out=-40, in=-140] (6, 0) to[out=-45, in=0] (6, -3) -- (0, -3) to[out=180, in=-135] (0, 0);
   \draw[thick, orange, fill=orange!5] (2, -1.5) -- (3.5, -1.5) arc (90:-90:0.6) -- (2, -2.7) to[out=180, in=-80] (0, 0) to[out=-45, in=180] (2, -1.5);
   \node[red] at (6.9, 3.2) {\small $C_1$};
   \node[blue] at (7.2, -3.2) {\small $N_{C_1,w'}$};
   \node[orange] at (4.7, -2) {\small $C_2$};
   \node[orange] at (5.4, -2.6) {\small $=N_{C_1,w}$};
   \node[draw, fill, inner sep=2pt, circle] (v1) at (3, 2) {};
   \node[draw, fill, inner sep=2pt, circle] (v2) at (3, -2) {};
   \node[above] at (v1) {\small $p_{C_1}$};
   \node[below] at (v2) {\small $p_{C_2}$};
   \node[draw, fill, inner sep=1pt, circle] (w1) at (0,0) {};
   \node[draw, fill, inner sep=1pt, circle] (w2) at (6,0) {};
   \node[left=1mm] at (0, 0) {\small $w$};
   \node[left=1mm] at (6, 0) {\small $w^\prime$};
   \draw[dashed] (v1) to[out=180, in=60] (0, 0) to[out=-60, in=180] (v2);
   \draw[dashed] (v1) to[out=0, in=120] (6, 0) to[out=-120, in=0] (v2);
   \end{scope}
   \begin{scope}[shift={(11,0)}]
   \draw[thick, red, fill=red!5] (0, 0) to[out=40, in=140] (6, 0) to[out=45, in=0] (6, 3) -- (0, 3) to[out=180, in=135] (0, 0);
   \draw[thick, blue, fill=blue!5] (0, 0) to[out=-40, in=-140] (6, 0) to[out=-45, in=0] (6, -3) -- (0, -3) to[out=180, in=-135] (0, 0);
   \draw[thick, orange, fill=orange!5] (2, -1.5) -- (4, -1.5) arc (90:-90:0.6) -- (2, -2.7) arc (270:90:0.6);
   \draw[thick, darkgreen, fill=darkgreen!5] (0, 0) to[out=20, in=90] (2, 0) to[out=-90, in=-20] (0, 0);
   \draw[thick, darkgreen, fill=darkgreen!5] (6, 0) to[out=20, in=90] (8, 0) to[out=-90, in=-20] (6, 0);
   \node[red] at (6.9, 3.2) {\small $C_1$};
   \node[orange] at (1, -2.3) {\small $C_2$};
   \node[darkgreen] at (2.3, -0.4) {\small $C_3$};
   \node[darkgreen] at (8.3, -0.4) {\small $C_4$};
   \node[blue] at (8.1, -3.2) {\small $N_{C_1,w}=N_{C_1,w'}$};
   \node[draw, fill, inner sep=2pt, circle] (v1) at (3, 2) {};
   \node[draw, fill, inner sep=2pt, circle] (v2) at (3, -2) {};
   \node[above] at (v1) {\small $p_{C_1}$};
   \node[below] at (v2) {\small $p_{C_2}$};
   \node[draw, fill, inner sep=1pt, circle] (w1) at (0,0) {};
   \node[draw, fill, inner sep=1pt, circle] (w2) at (6,0) {};
   \node[left=1mm] at (0, 0) {\small $w$};
   \node[left=1mm] at (6, 0) {\small $w^\prime$};
   \draw[dashed] (v1) to[out=180, in=60] (0, 0) to[out=-60, in=180] (v2);
   \draw[dashed] (v1) to[out=0, in=120] (6, 0) to[out=-120, in=0] (v2);
   \end{scope}
  \end{tikzpicture}
 \end{center}
 \caption{If there are two parallel edges in $G^\prime$, say from $C_1$ to $C_2$, added for $(C_1,w)$ and $(C_1, w^\prime)$,
 then the embedding of these two edges does not bound an area without any cycle in $\Lscr$:
 The case on the left, in which the edges are continued by nice paths in different cycles, cannot happen because one of them must contain the other and is therefore not in $\Nscr(C_1)$. In the other case, due to the minimality of $W(C_1)$, there must be other $\subseteq_{C_1}$-minimal neighbours $C_3$ at $w$ and $C_4$ at $w'$ as shown.
\label{fig:no_homotopic}}
\end{figure}

Thus, $G^\prime$ is a directed graph embedded in the same surface as $G$ and without pairs of homotopic edges.
By Lemma~\ref{lemma:degree5} there is some $C^*\in\Lscr_1$ with out-degree at most $6 g + 5$ in $G^{\prime}$; therefore $|W(C^*)|\le 6 g + 5$.
\end{proof}

\subsection{Tightness of the Efficient Cycle Lemma}

The constant $5$ in the Efficient Cycle Lemma for planar graphs is best possible:
let $G$ be a truncated dodecahedron, where the cycles in $\Lscr$ are given exactly by the decagonal faces, 
i.e., every cycle in $\Lscr$ is one-sided and touches exactly five neighbours, 
but there is no vertex contained in more than two of these cycles. 
Then the constructed graph $G^\prime$ is an icosahedron; all vertices have degree $5$.

However, the postulation that $W \subseteq C^*$ in the Efficient Cycle Lemma is not necessary for the proof of Theorem~\ref{thm:main_lpbased_planar}.
If this postulation is relaxed, the best lower bound for the constant in the planar version of the Efficient Cycle Lemma that we know is $4$:

\begin{lemma}
 There exists a planar graph $G$ and a laminar set $\Lscr$ of cycles in $G$ such that for every $C \in \Lscr$ and every $X\subseteq V(G)$ with $|X|\le 3$ there is a cycle in $\Lscr$ that shares a vertex with $C$ but contains no vertex of $X$.
\end{lemma}

\begin{proof}
 Let $G^\prime$ be the planar graph that arises from replacing each face of a cube by a $7 \times 7$ grid.
 Let $G$ be a planar graph that consists of edge-disjoint cycles $(C_v)_{v \in V(G^\prime)}$, each bounding a face on the sphere, such that $C_v$ and $C_w$ share exactly one vertex if and only if $\{v,w\} \in E(G^\prime)$.
 Let $\Lscr := \{C_v : v \in V(G^\prime)\}$.
 
 It is clear that except for the cycles corresponding to the corners of the cube, each cycle in $\Lscr$ has exactly four neighbours that are pairwise vertex-disjoint. Thus it is not possible to cover those with three vertices.
 For each of the corner cycles we add three two-sided neighbours as in Figure~\ref{fig:tightness_of_lemma} to $\Lscr$ (and $G$). 
 Now there are seven cycles sharing a vertex with a corner cycle, including the corner cycle itself, which means it is not possible to cover all of them by three vertices in $G$ as each such vertex covers only two of them. 
 A similar argument for the added two-sided cycles concludes the proof.
 \end{proof}
 
 \begin{figure}[htb]
 \begin{center} 
  \begin{tikzpicture}[scale=0.7]
   \draw[step=1.0,black,thin, dashed] (1,1) grid (7,7);
   \foreach \x in {1,...,7}
   {
    \foreach \y in {1,...,7}
    {
     \draw[thick,blue] (\x, \y) circle (0.5);
    }
   }
   \draw[thick,red] (1.36, 1.36) to[out=40,in=180] (2,1.5) -- (3.5,1.5) -- (3.5, 3.5) -- (1.5, 3.5) -- (1.5,2) to[out=-90,in=60] (1.36, 1.36);
   \draw[thick,red] (1.36, 6.64) to[out=-40,in=180] (2,6.5) -- (3.5,6.5) -- (3.5, 4.5) -- (1.5, 4.5) -- (1.5,6) to[out=90,in=-60] (1.36, 6.64);
   \draw[thick,red] (6.64, 1.36) to[out=140,in=0] (6,1.5) -- (4.5,1.5) -- (4.5, 3.5) -- (6.5, 3.5) -- (6.5,2) to[out=-90,in=120] (6.64, 1.36);
   \draw[thick,red] (6.64,6.64) to[out=220,in=0] (6,6.5) -- (4.5,6.5) -- (4.5, 4.5) -- (6.5, 4.5) -- (6.5,6) to[out=90,in=240] (6.64,6.64);
  \end{tikzpicture}
 \end{center}
 \caption{The black dashed lines show one of the six $7 \times 7$ grids $G^\prime$ is built of. For each grid vertex, $G$ contains one one-sided cycle in $\Lscr$, as shown by the blue cycles. For each corner cycle $C$ of the grid we add a two-sided cycle (red) that meets $C$ in a vertex that is not contained in any other neighbour of $C$. Since $C$ is adjacent to exactly three of those grids (three sides of the cube), it shares a vertex with exactly three one-sided and three two-sided cycles. Therefore it is not possible to cover $C$ and all its neighbours with three vertices.
\label{fig:tightness_of_lemma}}
\end{figure}
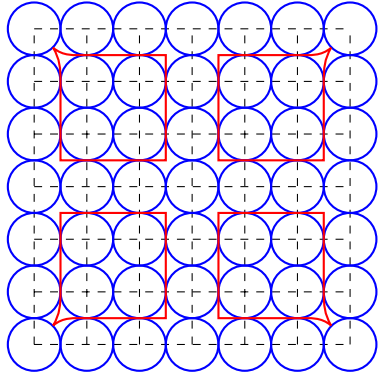

Closing the gap between the lower bound $4$ (this example) and the upper bound $5$ in the planar Efficient Cycle Lemma remains an open problem.

\subsection{Packing cycles in bounded-genus graphs}\label{section:packing_cycles_bounded_genus}

In this section we show how to obtain Theorem~\ref{thm:main_lpbased_boundedgenus}\,(a). 
The proof uses the Efficient Cycle Lemma~\ref{lemma:efficient_cycle_lemma_general} 
from the previous section and otherwise follows exactly the approach by~\cite{HuaMMV21}
for edge-disjoint packing of $D$-cycles. Therefore we only sketch the proof here.

First apply Lemma~\ref{lemma:compute_uncrossed_lp_solution} to obtain a near-optimal solution $x$ to the LP~\eqref{eq:lp}
with uncrossed support. 
Let $\Cscr^{\textnormal{sep}}$ denote the set of the separating cycles in $\Cscr$.
If $\sum_{C\in\Cscr^{\textnormal{sep}}}x_C \ge \frac{1}{2}\sum_{C\in\Cscr}x_C$, then set $x_C:=0$ for all $C\in\Cscr\setminus\Cscr^{\textnormal{sep}}$
and proceed with Lemma~\ref{lemma:efficient_cycle_lemma_general} 
(note that the support of $x$ is then laminar by Proposition~\ref{prop:uncrossed_separating_are_laminar}).
We end up with a vertex-disjoint set of at least $\frac{1}{6g+5}\sum_{C\in\Cscr^{\textnormal{sep}}}x_C$ cycles.

Otherwise partition the non-separating cycles in the support of $x$ into free homotopy classes $\Cscr_1,\ldots,\Cscr_k$.
By a theorem of Greene \cite{Gre19}, $k$ can be chosen to be $O(g^2\log g)$.
Take the class $\Cscr_i$ with the largest LP value $\sum_{C\in \Cscr_i}x_C$ and ignore other cycles. 
Exploiting a cyclic order of $\Cscr_i$, one can greedily round the LP solution restricted to that class 
to obtain a vertex-disjoint set $\Cscr_i^*\subset \Cscr_i$ with $|\Cscr_i^*| \ge \frac{1}{2}\sum_{C\in \Cscr_i}x_C$. 
See \cite{HuaMMV21} for details.

\subsection{The edge-disjoint version} \label{section:lprounding_edgedisjoint}

Again there is also an edge-disjoint version (b) of Theorem~\ref{thm:main_lpbased_boundedgenus}.
The high-level algorithm is the same. 
First we solve the edge-disjoint cycle packing LP~\eqref{eq:lp_edgedisjoint} and uncross the support of the solution.
This is still possible with a weight oracle due to Lemma~\ref{lemma:compute_uncrossed_lp_solution}.

The remaining proof is exactly the same, except that for the rounding step we need an edge-disjoint version of the Efficient Cycle Lemma:

\begin{lemma} \label{lemma:edgedisjoint_efficientcyclelemma}
Let $G$ be a graph embedded in an orientable surface of genus $g$, the sphere if $G$ is planar. Let $\Lscr$ be a non-empty laminar set of separating cycles in $G$.
Then there exists a one-sided cycle $C^* \in \Lscr$ and a set $W$ of its edges such that $|W| \leq 6 g + 5$ and every $C \in \Lscr$ is either edge-disjoint from $C^*$ or contains at least one edge from $W$.
\end{lemma}

One possibility to prove this is to adapt the proof of Lemma~\ref{lemma:efficient_cycle_lemma_general} by defining $W(C,N)$ for a cycle $C$ and an ``edge-neighbour'' $N$ (a cycle that shares an edge with $C$) to be the set of all edges in $E(C) \cap E(N)$.

However, it is also possible to deduce the edge-disjoint version of the Efficient Cycle Lemma from the vertex-disjoint version. 
For this, we first replace every edge $e$ in $G$ by a path $e_1, v_e, e_2$ of length two, 
and then replace each vertex $w$ of the original graph by a set $\{v_{w,C} : C \in \Lscr \text{ and } w \in V(C)\}$ of vertices. 
See Figure~\ref{fig:reduce_efficientcyclelemma_from_edge_to_vertexdisjoint}.
Since $\Lscr$ is laminar, we can still embed the edges $\{v_e, v_{w, C}\}$ for each $w \in e \in E(C)$ and each $C \in \Lscr$ on the orientable surface of genus $g$, 
and $\Lscr$ still induces a laminar family of cycles in the constructed graph. 
Therefore, in this instance there are a cycle and at most $6g+5$ of its vertices that hit all neighbours, 
which corresponds by construction to a cycle $C \in \Lscr$ and at most $6g+5$ of its edges that hit all edge-neighbours.

\begin{figure}[htb]
 \begin{center}
  \begin{tikzpicture}[scale=0.32, thick]
  \tikzstyle{vertex}=[circle,fill,minimum size=5,inner sep=0pt,outer sep=1pt]
  \tikzstyle{subdiv}=[circle,draw,minimum size=5,inner sep=0pt,outer sep=1pt]
  \tikzstyle{edge}=[very thick]
   \begin{scope}[shift={(0,0)}]
   \draw[edge, darkgreen] (-8, 5.92) -- (0, -0.08) -- (8, 5);
   \node[darkgreen] at (8, 3.6) {$C_2$};
   \draw[edge, violet] (-8, 6.08) -- (0, 0.08) -- (2, 8);
   \node[violet] at (3, 7) {$C_1$};
   \draw[edge, red] (-8, -3) -- (-0.08, 0) -- (-0.08, -8);
   \node[red] at (-7, -4) {$C_3$};
   \draw[edge, blue] (0.08, -8) -- (0.08, 0) -- (8, -4);
   \node[blue] at (7, -5) {$C_4$};
   \node[vertex] at (0,0) {};
   \node at (1.5,0) {$w$};
   \draw[black!50, ultra thick, ->] (12,0) -- (16,0);
   \end{scope}
   \begin{scope}[shift={(28,0)}]
   \node[subdiv, darkgreen] (v1) at (-6, 4.5) {};
   \node[subdiv, violet, dashed] (v1) at (-6, 4.5) {};
   \node[vertex, violet] (v2) at (-1, 3.5) {};
   \node[violet, anchor=north west] at (v2) {$v_{w, C_1}$};
   \node[subdiv, violet] (v3) at (1.8, 5) {};
   \node[vertex, darkgreen] (v4) at (0.5,1) {};
   \node[darkgreen, anchor=north west] at (v4) {$v_{w, C_2}$};
   \node[subdiv, darkgreen] (v5) at (6, 3) {};
   \node[subdiv, red] (v6) at (-6, -2.5) {};
   \node[vertex, red] (v7) at (-2, -2.5) {};
   \node[red, anchor=south] at (v7) {$v_{w, C_3}$};
   \node[subdiv, red] (v8) at (0, -6) {};
   \node[subdiv, blue, dashed] (v8) at (0, -6) {};
   \node[vertex, blue] (v9) at (2, -2.5) {};
   \node[blue, anchor=south] at (v9) {$v_{w, C_4}$};
   \node[subdiv, blue] (v10) at (6, -3.3) {};
   \draw[edge, violet] (-8, 7) -- (v1) -- (v2) -- (v3) -- (2, 8);
   \draw[edge, darkgreen] (-8, 5) -- (v1) -- (v4) -- (v5) -- (8, 6);
   \draw[edge, red] (-8, -4) -- (v6) -- (v7) -- (v8) -- (-1, -8);
   \draw[edge, blue] (1, -8) -- (v8) -- (v9) -- (v10) -- (8, -5);
   \end{scope}
  \end{tikzpicture}
 \end{center}
 \caption{In the original graph (left-hand side), four cycles (blue, green, red, violet) meet in the vertex $w$. 
 Since they do not cross in $w$, it is possible to split $w$ into four copies such that the cycles share a vertex
 only if they share an edge in the original instance. 
 These four copies are shown as filled vertices on the right; the unfilled circles result from subdividing the edges.\label{fig:reduce_efficientcyclelemma_from_edge_to_vertexdisjoint}}
\end{figure}
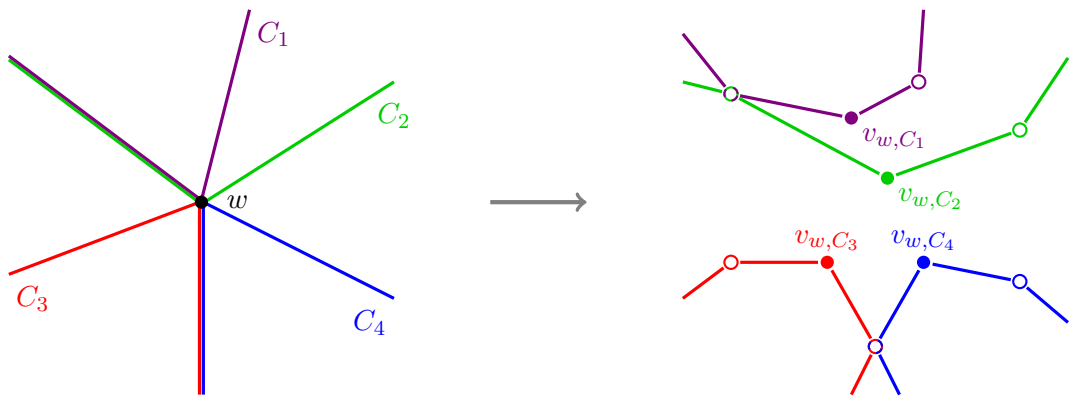
 
\section{Applications and Related Work}\label{sec:applications}

In this section we discuss related work and some applications of our theorems.
After reviewing some related work on cycle packing in Section~\ref{section:relatedwork_cyclepacking},
we discuss the relation to the ``dual'' cycle transversal problem in Section~\ref{section:erdosposa}.
In particular, our work combined with the classical work by \cite{GoeW98} implies that, 
for any uncrossable family of cycles in a planar graph, 
the maximum number of disjoint cycles is at most a constant factor times the minimum number
of vertices hitting all cycles.
Last but not least, in Sections~\ref{section:disjoint_paths} and~\ref{section:maxweightdisjointpaths}, 
we discuss the disjoint paths problem in more detail, including the weighted generalization.

\subsection{Related work on cycle packing}\label{section:relatedwork_cyclepacking}

Many variants of cycle packing problems have been studied, and this literature review cannot be complete.
In general graphs, the cycle packing problem is hard to approximate \cite{KriNSVY07,FriS11}.

Even in planar graphs, most variants are known to be NP-hard. 
An exception is the edge-disjoint packing of (arbitrary) cycles in planar digraphs, which is completely solved by
the famous Lucchesi--Younger theorem \cite{LucY78}.
For edge-disjoint packing in undirected planar graphs,
Caprara, Panconesi and Rizzi~\cite{CapPR04} devised a $2+\epsilon$-approximation algorithm for $\Cscr_{\textnormal{all}}$, 
Kr\'al' and Voss~\cite{KraV04} showed a 2-approximation algorithm for $\Cscr_{\textnormal{odd}}$,
and Garg, Kumar and Seb\H{o}~\cite{GarKS22} devised a 4-approximation algorithm for $\Cscr_D^{=1}$.
Theorem~\ref{thm:main_lpbased_planar}\,(b) extends the latter result to general uncrossable families.
Theorem~\ref{thm:main_combinatorial}\,(b) improves the approximation ratio to $3+\epsilon$;
see Table~\ref{table:overview_edgedisjoint}.

Vertex-disjoint packing seems to be more difficult. To the best of our knowledge,
constant-factor approximation algorithms and constant upper bounds on the integrality gap were only known in three cases:
For undirected planar graphs, a simple 3-approximation algorithm for $\Cscr_{\textnormal{all}}$
can be deduced from Euler's formula \cite{MaYZ13,CheFS12}, 
and Kr\'al\textquoteright{}, Sereni and Stacho \cite{KraSS12} devised a 
6-approximation algorithm for $\Cscr_{\textnormal{odd}}$.
For planar digraphs, Reed and Shepherd \cite{ReeS96} proved an upper bound of 28 on the integrality gap 
of \eqref{eq:lp} for $\Cscr_{\textnormal{all}}^{\rightarrow}$, 
which was improved to $16.31$ by Fox and Pach (see \cite{CamEM17}) 
and to $15.95$ by Cames van Batenburg, Esperet and M\"uller \cite{CamEM17}.
Theorem~\ref{thm:main_combinatorial}\,(a) yields an approximation ratio $3+\epsilon$ 
for any uncrossable family of cycles, and Theorem~\ref{thm:main_lpbased_planar}\,(a) 
implies an upper bound of 5 on the integrality gap; see Table~\ref{table:overview_vertexdisjoint}.

In bounded-genus graphs, the only previous result we are aware of is the constant-factor approximation algorithm 
for edge-disjoint packing of cycles in $\Cscr_D^{=1}$ \cite{HuaMMV21}.

 \begin{table}
 \centering
 \renewcommand{\arraystretch}{1.25}
 \begin{tabular}{| r || lr | lr || lr | lr || lr |}
 \hline
   & \multicolumn{4}{c ||}{Packing} & \multicolumn{4}{c ||}{Transversal} &  \multicolumn{2}{c |}{Erd\H{o}s--P\'osa} \\
  $\Cscr$ & \multicolumn{2}{c |}{approx} & \multicolumn{2}{c ||}{gap} & \multicolumn{2}{c |}{approx} & \multicolumn{2}{c ||}{gap} & \multicolumn{2}{c |}{ratio} \\
  \hline
  $\Cscr_{\textnormal{all}}^{\rightarrow}$ & $1$ & \cite{LucY78} & $1$ & \cite{LucY78} & $1$ & \cite{LucY78} & $1$ & \cite{LucY78} & $1$ & \cite{LucY78} \\
  $\Cscr_{\textnormal{all}}$ & $2 + \epsilon$ & \cite{CapPR04} & $[2,4]$ & \cite{MaYZ13} & $1$ & \!\![easy] & $2$ & \!\!\!\![easy] & $4$ & \!\!\!\!\!\!Kr\'al\textquoteright{}\,\cite{MaYZ13} \\
  $\Cscr_{\textnormal{odd}}$ & $2$ & \cite{KraV04} & $2$ & \cite{KraV04} & $1$ & \cite{Had75, EdmJ73} & $1$ & \cite{EdmJ73} & $2$ & \cite{KraV04} \\
  $\Cscr_D^{=1}$ & \newresult{$3+\epsilon$} {\scriptsize $<4$}\!\!\!\! & {\scriptsize\cite{GarKS22}} & $[2,4]$ & \cite{GarKS22} & $1+\epsilon$ & \cite{KleMZ15} & $[1.5,2]$ & \!\!\!\!\cite{HuaMMSV21,GarKS22} & $[2,4]$ & \cite{GarK20} \\
  any & \newresult{$3+\epsilon$} & & $[2,\,$\newresult{$4$}] & & $2.4$ & \cite{BerY12} & $[2,2.4]$ & \cite{BerY12} & $[4,\,$\newresult{$9.6$}] & \\
  \hline
 \end{tabular}
 \caption{State of the art for \emph{edge-disjoint} cycle packing and transversal of certain uncrossable families in \emph{planar} graphs. 
 The last row refers to a general uncrossable family of cycles 
 (see Section~\ref{section:examples_of_uncrossable} for more examples).
 The table shows the best known approximation ratios for cycle packing and cycle transversal
 and the known bounds on the integrality gaps of the LP~\eqref{eq:lp_edgedisjoint} and its dual \eqref{eq:dual_lp_edgedisjoint}.
 The last column shows the known bounds on the worst ratio of transversal and packing number.
 Results marked [easy] are easy because minimal feedback edge sets are spanning trees in the planar dual.
 The results from Theorem~\ref{thm:main_combinatorial}\,(b) and Theorem~\ref{thm:main_lpbased_planar}\,(b)
 are shown in bold blue; here we also show the previous best (for $\Cscr_D^{=1}$).
\label{table:overview_edgedisjoint}}
 \end{table}

 \begin{table}
 \renewcommand{\arraystretch}{1.25}
 \begin{tabular}{| r || lr | lr || lr | lr || lr |}
 \hline
   & \multicolumn{4}{c ||}{Packing} & \multicolumn{4}{c ||}{Transversal} &  \multicolumn{2}{c |}{Erd\H{o}s--P\'osa} \\
  $\Cscr$ & \multicolumn{2}{c |}{approx} & \multicolumn{2}{c ||}{gap} & \multicolumn{2}{c |}{approx} & \multicolumn{2}{c ||}{gap} & \multicolumn{2}{c |}{ratio} \\
  \hline
  $\Cscr_{\textnormal{all}}^{\rightarrow}$ & \newresult{$3+\epsilon$} & & $[2,$\,\newresult{$5$}] {\scriptsize$\subset [2,15.95]$}\!\!\!\!\!\!\!\!\!\!\!\!\! & {\scriptsize\cite{CamEM17}} & $2.4$ & \cite{BerY12} & $[1.5,2.4]$ & \cite{BerY12} & $[2,$\,\newresult{$12$}] {\scriptsize$\subset [2,38.28]$}\!\!\!\!\!\!\!\!\!\!\!\!\!\!\!\!\!\! & \\
  $\Cscr_{\textnormal{all}}$ & $3$ & \!\!\cite{CheFS12, MaYZ13} & $[1.5,3]$ & \cite{CheFS12, MaYZ13} & $1 + \epsilon$\! & \!\!\cite{KleK01} & $[1.5,2.4]$ & \cite{BerY12} & $[2, 3]$ & \!\!\!\!\cite{CheFS12,MaYZ13} \\
  $\Cscr_{\textnormal{odd}}$ & \newresult{$3+\epsilon$} {\scriptsize $<6$}\!\!\!\!\!\!\!\!\! & {\scriptsize\cite{KraSS12}} & $[2,$\,\newresult{$5$]} {\scriptsize$\subset [2,6]$}\!\!\!\!\!\! & {\scriptsize\cite{KraSS12}} & $2.4$ & \cite{BerY12} & $ [1.5,2.4]$ & \cite{BerY12} & $[2,6]$ & \cite{KraSS12} \\
  $\Cscr_D^{=1}$ & \newresult{$3+\epsilon$} & & $[2,$\,\newresult{$5$}] & & $2.4$ & \cite{BerY12} & $[1.5,2.4]$ & \cite{BerY12} & $[2,$\,\newresult{$12$}] & \\
  any & \newresult{$3+\epsilon$} & & $[2,$\,\newresult{$5$}] & & $2.4$ & \cite{BerY12} & $[2,2.4]$ & \cite{BerY12} & $[2,\,$\newresult{$12$}] & \\
  \hline
 \end{tabular}
 \caption{State of the art for \emph{vertex-disjoint} cycle packing and transversal of certain uncrossable families in \emph{planar} graphs. 
 The last row refers to a general uncrossable family of cycles. 
 The table shows the best known approximation ratios for cycle packing and cycle transversal
 and the known bounds on the integrality gaps of the LP~\eqref{eq:lp} and its dual \eqref{eq:dual_lp}.
 The last column shows the known bounds on the worst ratio of transversal and packing number.
 The results from Theorem~\ref{thm:main_combinatorial}\,(a) and Theorem~\ref{thm:main_lpbased_planar}\,(a)
 are shown in bold blue; here we also show the previous best.
  \label{table:overview_vertexdisjoint}}
 \end{table}
 
 \subsection{Cycle transversal and the Erd\H{o}s--P\'osa property}\label{section:erdosposa}
 
The dual LP to the edge-disjoint cycle packing LP~\eqref{eq:lp_edgedisjoint} is 
\begin{equation} \label{eq:dual_lp_edgedisjoint}
\min \left\{ \sum_{e\in E}y_e : \sum_{e\in C}y_e\ge 1 \ (C\in\Cscr),\ y_e\ge 0 \ (e\in E) \right\}.
\end{equation}
The problem of finding an optimum integral solution to this LP has different names,
depending in $\Cscr$: e.g., for $\Cscr_{\textnormal{all}}$ one speaks of the (unweighted) feedback edge set problem.
Many of these problems are well-understood in planar graphs; see Table~\ref{table:overview_edgedisjoint}.

As for cycle packing, the version where we need to hit cycles by vertices instead of edges is more difficult. 
The problem of finding an optimum integral solution to \eqref{eq:dual_lp} 
(the dual of the vertex-disjoint cycle packing LP~\eqref{eq:lp}) is known as the 
(unweighted) \emph{cycle transversal problem}, sometimes also called the \emph{hitting cycle problem}.
Table~\ref{table:overview_vertexdisjoint} shows the state of the art. 
Most notably, Goemans and Williamson~\cite{GoeW98} devised a primal-dual 3-approximation algorithm 
for finding a minimum transversal of an uncrossable family $\Cscr$.
The approximation ratio was later improved to $2.4$ by Berman and Yaroslavtsev~\cite{BerY12}
(see also \cite{Yar14}), and this also shows an upper bound of $2.4$ on the integrality gap of~\eqref{eq:dual_lp}.
The integrality gap is at least 2 in general: e.g., for $\Cscr_D^{\ge 1}$, if
$G$ is a grid with $k$ columns and $2k$ rows and $D$ is the set of vertical edges between the two middle rows, then
setting $x_v=\frac{1}{4}$ for all $2k$ vertices $v$ in the two middle rows is an LP solution, 
but any transversal has at least $k-1$ vertices.
Most of the other lower bounds in the tables result from very simple variants of $K_4$ or an octahedron, otherwise we give a reference.

The algorithms by Goemans and Williamson and by Berman and Yaroslavtsev
actually work for the weighted transversal problem (with nonnegative vertex weights).
The edge transversal problem can be reduced to the weighted vertex transversal problem by giving all original vertices a very high weight and subdividing every edge by a vertex of weight 1. See the last row of Table~\ref{table:overview_edgedisjoint}.
 
Erd\H{o}s and P\'osa \cite{ErdP65} showed that there is a function $f$ such that for every undirected graph $G$
that has no more than $k$ pairwise vertex-disjoint cycles, 
there is a cycle transversal (also called feedback vertex set) of size at most $f(k)$.
One also says that cycles in undirected graphs have the Erd\H{o}s--P\'osa property. 
Odd cycles \cite{Ree99,RauR01} or $D$-cycles \cite{GarVY97} do not have this property, 
even for graphs embedded in the projective plane, but cycles in digraphs do \cite{ReeRST96}.

However, our result implies that in planar graphs  
we do have the Erd\H{o}s--P\'osa property for any uncrossable set of cycles, 
and we even have a constant upper bound on the ratio. 
Namely, the cycle transversal algorithms~\cite{GoeW98,BerY12} also show an uppper bound on the integrality gap 
of the dual LP \eqref{eq:dual_lp} by comparing the transversal to a fractional primal solution. 
Combining this bound of 2.4 for the cycle transversal problem \cite{BerY12} with our 
Theorem~\ref{thm:main_lpbased_planar}, we obtain:

\begin{corollary}\label{cor:primaldual}
For a planar graph $G$ and an uncrossable family $\Cscr$ of cycles, 
there exist feasible integral solutions $x$ and $y$ to \eqref{eq:lp} and \eqref{eq:dual_lp}
such that $\sum_{v\in V}y_v \le 12\cdot \sum_{C\in\Cscr} x_C$. 
\hfill\qed
\end{corollary}

In other words, the minimum cardinality of a cycle transversal is at most 12 times the maximum cardinality of a cycle packing.

For the special case $\Cscr_\textnormal{odd}$, Kr\'al\textquoteright{}, Sereni and Stacho \cite{KraSS12} 
showed a constant of 6, improving on the work of Fiorini et al.\ \cite{FioXX07}. 
(A constant was also shown for odd cycle packing in highly connected \cite{Tho01,RauR01} and in bounded-genus graphs \cite{KawN07}.)
Again in planar graphs, for $\Cscr_\textnormal{all}$ a constant of 3 can be shown rather easily \cite{CheFS12,MaYZ13}.
For $\Cscr_\textnormal{all}^{\rightarrow}$, however, the previous best bound was $38.28$, which results from
multiplying the upper bound 2.4 \cite{BerY12} on the integrality gap of \eqref{eq:dual_lp} with
the upper bound 15.95 \cite{CamEM17} (which we improve to 5) on the integrality gap of \eqref{eq:lp}. 
For other uncrossable families of cycles it seems that no constant bound was known.
See Tables~\ref{table:overview_edgedisjoint}~and~\ref{table:overview_vertexdisjoint} 
for an overview of the known results for planar graphs.

\subsection{The disjoint paths problem}\label{section:disjoint_paths}
 
An instance of the \emph{maximum vertex-disjoint paths problem} consists of two undirected graphs on the same vertex set:
a \emph{supply} graph $G=(V,S)$ and a \emph{demand} graph $H=(V,D)$. 
The task is to find a maximum number of pairwise vertex-disjoint (simple) cycles in $G+H$ 
such that each of these cycles contains exactly one demand edge (in otherwords, $D$-cycles). 
So a demand edge specifies a pair of endpoints and asks for a path between these endpoints in the supply graph,
and we want to satisfy as many demands as possible.
This is one of the classical combinatorial optimization problems, with many applications; see, e.g.,
\cite{KorLPS90} and Part VII of \cite{Sch03}.

The best known approximation ratio for the maximum vertex-disjoint paths problem is $O(\sqrt{n})$, where $n=|V|$; 
this is obtained by greedily picking a shortest path that satisifies some demand.
For the case when $G$ is planar, the approximation ratio was improved to $\tilde O(n^{9/19})$ by \cite{ChuKL16}.
The situation is similar for the \emph{maximum edge-disjoint paths problem} 
(in which the cycles are only required to be pairwise edge-disjoint). 
Here the best known approximation ratio is $O(\sqrt{n})$ \cite{CheKS06}, even if $G$ is planar. 

Recently, constant-factor approximation algorithms were devised for \emph{fully planar} instances of the
maximum edge-disjoint paths problem, i.e., when $G+H$ is planar 
(Huang et al.~\cite{HuaMMSV21}, Garg, Kumar and Seb\H{o}~\cite{GarKS22});
the best known approximation ratio is 4 \cite{GarKS22}.
Even in this special case it is NP-complete to decide whether all demands can be satisfied, 
both in the edge-disjoint and in the vertex-disjoint version, 
and even when in addition $G+H$ is cubic \cite{MidP93}, i.e., all vertices have degree 3.
Even more recently, Huang et al.~\cite{HuaMMV21} devised a constant-factor approximation algorithm for instances $(G,H)$ where $G+H$
can be embedded in an orientable surface of constant genus.

Middendorf and Pfeiffer~\cite{MidP93} showed that
the (maximum) edge-disjoint paths problem can be reduced to the (maximum) vertex-disjoint paths problem;
this reduction is approximation-preserving and embeds the resulting graph $G'+H'$ on the same surface as the original graph $G+H$.
Hence it is natural to ask whether the maximum vertex-disjoint paths problem also admits a constant-factor approximation algorithm for fully planar or bounded-genus instances.
For fully planar instances, Corollary~\ref{cor:vdp_edp_planar}, 
which is an immediate consequence of Theorem~\ref{thm:main_combinatorial} applied to $D$-cycles in $G+H$ 
using Propositions~\ref{prop:uncrossable} and~\ref{prop:compute_irrelevant_edges},
answers this question affirmatively 
by giving a $(3+\epsilon)$-approximation algorithm for any $\epsilon>0$.
It also improves the constant from 4 \cite{GarKS22} to $3+\epsilon$ in the edge-disjoint case.

By the Middendorf--Pfeiffer reduction~\cite{MidP93} from the edge-disjoint to the vertex-disjoint paths problem,
the former cannot be harder to approximate. The converse does not seem to be true.
Although Middendorf and Pfeiffer~\cite{MidP93} also present a reduction from 
the fully planar vertex-disjoint paths problem to the special case in which every vertex has degree 3
(in which case edge-disjoint and vertex-disjoint coincide),
this reduction is not approximation-preserving (it adds new demand edges). 

In particular, both previous constant-factor approximation algorithms for the edge-disjoint case 
do not seem to extend easily to the vertex-disjoint case. As said in Section~\ref{section:gks},
the approach by \cite{GarKS22} relies on the fact that if three cycles of a laminar set of cycles share the same edge, 
then at least two of these cycles must be comparable in the $\subseteq_{\infty}$-relation.
This is certainly not true for a laminar set of cycles sharing the same vertex:
arbitrarily many cycles can share a vertex although their interiors are disjoint.
The algorithm by \cite{HuaMMSV21} first computes a subset of demands such that the cut criterion is satisfied,
by considering the equivalent cycle criterion in the planar dual. The key problem in this paper (called \textsc{NonNegativeCycles} there)
does not seem to be useful for the vertex-disjoint case.

Although we obtain the same approximation guarantee $3+\epsilon$ now, the vertex-disjoint paths problem
may actually be harder than the edge-disjoint paths problem. 
As an indication, in the approximate max-flow min-cut theorem that we obtain for
vertex-disjoint paths, the constant is much worse than in the edge-disjoint case:

\begin{corollary}
Let $(G,H)$ be a fully planar instance of the maximum vertex-disjoint paths problem.
Then the minimum number of vertices whose removal destroys all $s$-$t$-paths in $G$ for all $\{t,s\}\in E(H)$
is at most 12 times the maximum number of demand edges for which there are vertex-disjoint paths.
\end{corollary}

\begin{proof}
Directly from Corollary~\ref{cor:primaldual} applied to the $D$-cycles in $G+H$,
using Proposition~\ref{prop:uncrossable}.
\end{proof}

For bounded-genus instances, another result of~\cite{HuaMMV21}, 
exploiting a theorem of Przytycki~\cite{Prz15}, can also be transferred immediately:

\begin{corollary}\label{cor:g_square_boundedgenus}
For any fixed $g$, there is a polynomial-time $O(g^2)$-approximation algorithm for the 
maximum vertex-disjoint paths problem in instances $(G,H)$ for which
$G+H$ is embedded in an orientable surface of genus~$g$.
\end{corollary}

\begin{proof}
We again start by computing an uncrossed LP solution $x$ and proceed as in Section~\ref{section:packing_cycles_bounded_genus} 
if the separating cycles contribute at least half of the LP value. 

In the other case, consider only non-separating cycles in the support of $x$.
Let $G$ be the graph whose vertices are (representatives of) the free homotopy classes of these cycles 
and whose edges correspond to representative cycles that cross (once). 
Then $G$ has a vertex of degree $O(g^2)$ \cite{Prz15} and hence can be colored with $O(g^2)$ colors. 
By choosing the color class with the largest contribution to the LP value and 
ignoring the two extreme cycles in each free homotopy class of this color, 
one can apply a greedy rounding algorithm to all these homotopy classes independently. 
This yields an $O(g^2)$-approximation algorithm for cycle packing unless the maximum number of disjoint $D$-cycles is $O(g^2 \log g)$.

Finally, if the maximum number of disjoint $D$-cycles is bounded by a constant $K$,
we can enumerate all subsets of at most $K$ demand edges and apply the well-known result 
of Robertson and Seymour~\cite{RobS95} to solve the maximum vertex-disjoint paths problem optimally.
\end{proof}

\subsection{Maximum-weight disjoint paths} \label{section:maxweightdisjointpaths}

There is also a weighted version of the disjoint paths problem, called
the \emph{maximum-weight vertex-disjoint} (or \emph{edge-disjoint}) \emph{paths problem}:
given a supply graph $G=(V,S)$, a demand graph $H=(V,D)$, and 
nonnegative weights $w \colon D \to \mathbb{R}_{\ge 0}$ of the demand edges, 
find a set of pairwise vertex-disjoint (or edge-disjoint) $D$-cycles in $G+H$, maximizing the total weight of their demand edges. 

The $O(\sqrt{n})$-approximation algorithms for the general vertex-disjoint and edge-disjoint paths problem still work in this more general setting, 
but constant-factor approximation algorithms are known only for very restricted versions of the maximum-weight edge-disjoint paths problem. 
In particular, Chekuri, Mydlarz and Shepherd~\cite{CheMS07} devised a 4-approximation algorithm 
for the case when $G$ results from a tree by duplicating edges, 
and Naves, Shepherd and Xia~\cite{NavSX22} obtained a 224-approximation algorithm for outerplanar graphs $G$.
For vertex-disjoint paths, no such results seem to be known at all.

We will show that both Theorem~\ref{thm:main_lpbased_planar} and Theorem~\ref{thm:main_lpbased_boundedgenus} 
can be extended to the maximum-weight disjoint paths problem if $G+H$ is planar 
(or embedded in an orientable surface of bounded genus) without any loss in the approximation guarantee. 
For the edge-disjoint case this follows almost immediately from the work of \cite{GarKS22} and \cite{HuaMMV21}; 
for the vertex-disjoint case we will need to combine our Efficient Cycle Lemma with the fractional local ratio method.

All these results will also bound the integrality gaps of the weighted versions of LPs~\eqref{eq:lp_edgedisjoint} and ~\eqref{eq:lp}. 
Let $\Cscr$ be the family of $D$-cycles in $G+H=(V,E)$ and define $w(C) := w(C \cap D)$ for any $C \in \Cscr$. 
Then the maximum-weight vertex-disjoint and edge-disjoint paths LPs are given by:
 \begin{equation}\label{eq:lp_weighted}
\max \left\{ \sum_{C\in\Cscr}w(C) x_C : \sum_{C\in\Cscr: v\in C}x_C\le 1 \ (v \in V),\ x_C\ge 0 \ (C\in\Cscr) \right\}
\end{equation}
 \begin{equation}\label{eq:lp_weighted_edgedisjoint}
\max \left\{ \sum_{C\in\Cscr}w(C) x_C : \sum_{C\in\Cscr: e\in C}x_C\le 1 \ (e \in E),\ x_C\ge 0 \ (C\in\Cscr) \right\}
\end{equation}

For the vertex-disjoint approach we use the fractional local ratio method in a similar way as Chan and Lau~\cite{ChaL12} 
did for the hypergraph matching problem. The following result is implicit in their paper:

\begin{theorem}[\cite{ChaL12}]\label{thm:local_ratio_method}
 Let $\Gscr = (\Vscr, \Escr)$ be a hypergraph and $x \colon \Escr \to \mathbb{R}_{\geq0}$ 
 such that $x(\{e \in \Escr : v \in e\}) \leq 1$ for all $v \in \Vscr$. 
 For any $e \in \Escr$ define $N[e] := \{e^\prime \in \Escr : e \cap e^\prime \neq \emptyset\}$. 
 Let further $k \in \mathbb{N}$ and $\mathcal{E} = \{e_1, \dots, e_m\}$ such that 
 $x(N[e_i] \cap \{e_i, e_{i+1}, \dots, e_m\}) \leq k$ for all $i=1,\ldots,m$.
 
 Then for any given edge weights $w \colon \Escr \to \mathbb{R}_{\ge 0}$ one can find a set $M \subseteq \Escr$ 
 of pairwise disjoint hyperedges such that $w(M) \geq \frac{1}{k} \sum_{e\in\mathcal{E}}w(e)x(e)$ in polynomial time.
\end{theorem}

The algorithm works by by (i) finding the hyperedge $e_i$ with smallest index such that $w(e_i)$ and $x(e_i)$ are both positive, 
(ii) modifying the weights by subtracting $w(e_i)$ from the weight of all edges in $N[e_i]$,
(iii) applying the algorithm recursively to the resulting instance, 
(iv) and adding $e_i$ to the resulting solution if this is feasible.
The approximation guarantee follows easily from the induction hypothesis and the fact 
that the returned solution contains at least one element of $N[e_i]$.
This has been called the fractional local ratio method; see also \cite{LauRS11} for details.

Since the disjoint paths problem can be formulated as a hypergraph matching problem where each hyperedge corresponds to a $D$-cycle, 
this can be used to obtain constant-factor approximation algorithms for the maximum-weight disjoint paths problem.
The order of the hyperedges that is needed in Theorem~\ref{thm:local_ratio_method} can be established with the Efficient Cycle Lemma.

\begin{theorem}\label{thm:weighted_disjoint_paths_planar}
Let $G=(V,S)$ and $H=(V,D)$ such that $G+H$ is planar and $w \colon D \to \mathbb{R}_{\geq 0}$. Let $\Cscr$ be the family of $D$-cycles in $G+H$. Then we can compute a
\begin{itemize}
\item[(a)] vertex-disjoint subset of $\Cscr$ whose weight is at least $\frac{1}{5}$ the value of the LP~\eqref{eq:lp_weighted}
\item[(b)] edge-disjoint subset of $\Cscr$ whose weight is at least $\frac{1}{4}$ the value of the LP~\eqref{eq:lp_weighted_edgedisjoint}
\end{itemize}
in polynomial time.
\end{theorem}
\begin{proof}
 First we solve the LP~\eqref{eq:lp_weighted} or \eqref{eq:lp_weighted_edgedisjoint}, respectively.
 The separation problem of the dual can still be solved by $|D|$ calls to Dijkstra's shortest path algorithm in $G$, 
 so the LP can be solved similarly to \eqref{eq:lp} or \eqref{eq:lp_edgedisjoint} (cf.\ Proposition~\ref{prop:solve_lp}).
 Alternatively, there exists an equivalent polynomial-size LP formulation that uses flow variables $f_{e,d}$ 
for $d \in D$ and $e \in S \cup D$, indicating which fraction of edge $e$ is used by $D$-cycles that contain $d$.

Next, we uncross the LP solution. This can be done similarly to the unweighted case (cf.\ the proof of Theorem~\ref{thm:compute_laminar_lp_solution}):
the weighted version of LP~\eqref{eq:lp_minimize_cycle_length} can still be solved similarly to the maximum-weight disjoint paths LP, 
and it is easy to see that any uncrossing step does not change the total weight of the fractional solution.
Therefore, we get an optimum LP solution to \eqref{eq:lp_weighted} or \eqref{eq:lp_weighted_edgedisjoint} with laminar support $\Lscr$.
From this point on, the proofs for (a) and (b) are different:

For (a), we use Theorem~\ref{thm:local_ratio_method}.
We can view the laminar LP solution as a fractional hypergraph matching in a hypergraph $\Gscr := (\Vscr, \Escr)$:
Let $\Vscr := V$ and $\Escr := \{V(C) : C \in \Lscr\}$.
For $l=1, \dots, |\Lscr|$ we know from Lemma~\ref{lemma:efficient_cycle_lemma_planar}
that there exists some $e_l \in \Escr$ such that 
$N[e_l] \setminus \{e_1, \dots, e_{l-1}\}$ can be covered by at most five vertices 
and thus $x(N[e_l] \setminus \{e_1, \dots, e_{l-1}\}) \leq 5$. 
Since we can find such an $e_l$ in polynomial time by complete enumeration, 
we find an order $\Escr = \{e_1, \dots, e_{|\Lscr|}\}$ as in the setting of Theorem~\ref{thm:local_ratio_method}. 
Therefore we can apply Theorem~\ref{thm:local_ratio_method} with $k=5$ and weights given by $w(C \cap D)$.

For (b), we could proceed similarly, using Lemma~\ref{lemma:edgedisjoint_efficientcyclelemma}, but this would only yield $\frac{1}{5}$ instead of $\frac{1}{4}$.
Instead, we just observe that the proof by Gark, Kumar and Seb\H{o} \cite{GarKS22} (cf. Section \ref{section:gks}) still works: 
the weighted version of LP~\eqref{eq:lp_rounding_gks} is still integral, therefore we get a half-integral solution 
to~\eqref{eq:lp_weighted_edgedisjoint} of at least half the optimum value. Since the conflict graph is planar and hence four-colorable,
we can partition all cycles in the support of that half-integral LP solution into four edge-disjoint families, one of which must have enough weight.
\end{proof}

Note that once we get a maximum-weight fractional cycle packing with laminar support, 
the approach works for arbitrary cycle families $\Cscr$, not only for $D$-cycles.
However, finding such a fractional solution seems to be much more complicated for arbitrary (weighted) uncrossable cycle families; 
in particular the weights need to be incorporated in the weight oracle, and some structure is needed to make uncrossing work.

However, the approach can be extended to the case where $G+H$ is embedded in an orientable surface of bounded genus. 
Similar to the the proof of Theorem~\ref{thm:weighted_disjoint_paths_planar} we obtain 
with the bounded-genus version of the Efficient Cycle Lemma (Lemma~\ref{lemma:efficient_cycle_lemma_general}):

\begin{theorem}\label{thm:weighted_disjoint_paths_bounded_genus}
Let $G=(V,S)$ and $H=(V,D)$ such that $G+H$ can be embedded in a fixed orientable surface of genus $g$ and $w \colon D \to \mathbb{R}_{\geq 0}$. 
Let $\Cscr$ be the family of $D$-cycles in $G+H$. Then we can compute a
\begin{itemize}
\item[(a)] vertex-disjoint subset of $\Cscr$ whose weight is
$\Omega(\frac{1}{g^2\log g})$ times the value of the LP~\eqref{eq:lp_weighted}
\item[(b)] edge-disjoint subset of $\Cscr$ whose weight is
$\Omega(\frac{1}{g^2\log g})$ times the value of the LP~\eqref{eq:lp_weighted_edgedisjoint}
\end{itemize}
in polynomial time.
\end{theorem}
\begin{proof}
 We can find an optimum solution to the LP~\eqref{eq:lp_weighted} or \eqref{eq:lp_weighted_edgedisjoint} as in Theorem~\ref{thm:weighted_disjoint_paths_planar}.
 Since single uncrossing steps do not change the total weight, we can still find a near-optimum LP solution $x$ 
 with uncrossed support similar to Lemma~\ref{lemma:compute_uncrossed_lp_solution}.
 If the separating cycles contribute more than half to the LP value, 
 we get an $O(g)$-approximation similarly to Theorem~\ref{thm:weighted_disjoint_paths_planar}.
 Otherwise we pick the free homotopy class that carries most of the (weighted) LP value.
 Huang et al.~\cite{HuaMMV21} showed that the cycles in this homotopy class can be ordered $\Cscr^* = \{C_1, \dots, C_k\}$ 
 such that for each $z \in V(G+H) \cup E(G+H)$ the cycles containing $z$ are given as an interval 
 $\{C_i, C_{i+1}, \dots, C_j\}$ or $\Cscr^* \setminus \{C_i, \dots, C_j\}$ for some $1 \leq i \leq j \leq k$. 
 Therefore, we can pick $l:=\max\{1,\lfloor x(\Cscr^*) \rfloor\} > \frac{1}{2}x(\Cscr^*)$ cycles from $\Cscr^*$:
 for any offset $p$ with $0 \le p< x(\Cscr^*)$, all cycles in
 \begin{equation} \label{eq:pickcyclesincyclicorder}
  \left\{C_i \in \Cscr^* : \sum_{j=1}^{i-1} x_{C_j} \leq (p + k) \bmod x(\Cscr^*) < \sum_{j=1}^i x_{C_j} 
  \text{ for some } k \in \{0,\ldots,l-1\} \right\}
 \end{equation}
 are pairwise vertex-disjoint or edge-disjoint, respectively.
 Choosing $p$ uniformly at random, we pick any cycle $C_j$ in \eqref{eq:pickcyclesincyclicorder} with probability 
 $\frac{l\cdot x_{C_j}}{x(\Cscr^*)} > \frac{1}{2}x_{C_j}$. 
 Thus, for the best offset we get an integral solution of weight $\Omega(\frac{1}{g^2\log g})$ times the value of the LP.
\end{proof}

\section{Conclusion}

We devised general approximation algorithms for packing cycles in an uncrossable family in planar and bounded-genus graphs.
These imply new results in several well-studied special cases, most notably the first constant-factor
approximation algorithm for the vertex-disjoint paths problem in fully planar instances, and even a weighted version
in bounded-genus instances.
For planar graphs, our results complement previous work on the cycle transversal problem and thus yield new
approximate min-max theorems. 

Several questions remain open: besides improving the constants in the planar case and determining the exact integrality gaps
(see Tables~\ref{table:overview_edgedisjoint}~and~\ref{table:overview_vertexdisjoint}), one would hope for
a constant-factor approximation algorithm for cycle transversal of uncrossable families in bounded-genus graphs. 
Some tools developed in this paper may be helpful towards this goal.

\subsection*{Acknowledgement}
The authors would like to thank Mohit Singh for pointing Theorem~\ref{thm:local_ratio_method} out to us.

\bibliographystyle{plain}
\bibliography{bibliography}

 \end{document}